\documentclass[11pt]{article}
\usepackage{amssymb}
\usepackage{amsthm}
\usepackage{amsmath}
\usepackage[dvips]{epsfig}
\usepackage{graphicx}
\usepackage{color}
\usepackage{url}
\usepackage{tikz}

\makeatletter\@addtoreset {equation}{section}\makeatother

\usetikzlibrary{matrix,arrows}
\newtheorem{thm}{Theorem}[section]
\newtheorem{cor}[thm]{Corollary}
\newtheorem{prop}[thm]{Proposition}
\newtheorem{lem}[thm]{Lemma}

\newtheorem{remark}{Remark}

\DeclareMathOperator{\sech}{sech}
\DeclareMathOperator{\Range}{Range}
\DeclareMathOperator{\Span}{span}

\begin{document}

\title{\bf $L^2$ orbital stability of Dirac solitons \\ in the massive Thirring model}

\author{Andres Contreras, Dmitry E. Pelinovsky, and Yusuke Shimabukuro \\
{\small \it Department of Mathematics and Statistics, McMaster
University, Hamilton, Ontario, Canada, L8S 4K1 } }

\date{\today}
\maketitle

\begin{abstract}
We prove $L^2$ orbital stability of Dirac solitons in the massive
Thirring model. Our analysis uses local well posedness
of the massive Thirring model in $L^2$, conservation of the charge functional,
and the auto--B\"{a}cklund transformation. The latter transformation
exists because the massive Thirring model is integrable via the inverse scattering
transform method.
\end{abstract}



\section{Introduction}

Among one-dimensional nonlinear Dirac equations, the Massive Thirring Model (MTM) is
particularly interesting because of its integrability via the inverse scattering transform method \cite{KN,KM}.
The nonlinear Dirac system arises as a relativistic extension of the nonlinear Schr\"{o}dinger equation
and while they share common features, the analysis of the Dirac system is generally more involved
because the classical energy-based methods \cite{GSS} do not apply.

We consider the Cauchy problem of the MTM system
\begin{align} \label{MTM}
\left\{ \begin{array}{l}
i(u_t+u_x)+v+u|v|^2=0, \\
i(v_t-v_x)+u+v|u|^2=0, \end{array} \right.
\end{align}
subject to an initial condition $(u,v)|_{t = 0} = (u_0,v_0)$ in $H^s(\mathbb{R})$ for $s\geq 0$.

The Cauchy problem for the MTM system is known to be locally well-posed in $H^s(\mathbb{R})$ for $s>0$
and globally well-posed for $s>\frac{1}{2}$ \cite{SelTes} (see earlier results in \cite{Del}).
More pertinent to our study is the global well-posedness in $L^2(\mathbb{R})$ proved in the recent work
of Candy (see Theorem 1 in \cite{Candy} for $s = 0$). The next theorem summarizes the global well-posedness results
for the scopes needed in our work.

\begin{thm}{\cite{Candy}}
Let $(u_0,v_0) \in H^m(\mathbb{R})$ for an integer $m \geq 0$. There exists a unique global solution
$(u,v)\in C(\mathbb{R}; H^m(\mathbb{R}))$ to the MTM system \eqref{MTM} such that the charge is conserved
\begin{equation}
\label{l2-conservation}
\|u(\cdot,t)\|_{L^2}^2+\|v(\cdot,t)\|_{L^2}^2=\|u_0\|_{L^2}^2+\|v_0\|_{L^2}^2
\end{equation}
for every $t\in \mathbb{R}$. Moreover, the solution depends continuously on initial data $(u_0,v_0) \in H^m(\mathbb{R})$.
\label{theorem-Candy}
\end{thm}

We are interested in orbital stability of  MTM solitons given by the explicit expressions
\begin{eqnarray}\label{onesol}
\left\{ \begin{array}{l}
u_{\lambda}(x,t) = i \delta^{-1} \sin(\gamma) \; \text{sech}\left[\alpha (x + \nu t)-i \frac{\gamma}{2} \right] \; e^{-i \beta (t + \nu x)}, \\
v_{\lambda}(x,t) = -i\delta \sin(\gamma) \; \text{sech}\left[\alpha (x + \nu t)+i \frac{\gamma}{2} \right] \; e^{-i \beta (t + \nu x)},
\end{array} \right.
\end{eqnarray}
where $\lambda$ is an arbitrary complex nonzero parameter that determines $\delta=|\lambda|$, $\gamma=2\mbox{Arg}(\lambda)$,
as well as
$$
\nu = \frac{\delta^{2} - \delta^{-2}}{\delta^{2} + \delta^{-2}}, \quad
\alpha = \frac{\delta^{2} + \delta^{-2}}{2} \sin \gamma, \quad
\beta = \frac{\delta^{2} + \delta^{-2}}{2} \cos \gamma.
$$
Let us now state the main result of our work.

\begin{thm} \label{maintheorem}
Let $(u,v)\in C(\mathbb{R};L^2(\mathbb{R}))$ be a solution of the MTM system (\ref{MTM})
in Theorem \ref{theorem-Candy} and $\lambda_0$ be a complex non-zero number. There exist a real positive
constant $\epsilon$ such that if the initial value $(u_0, v_0)\in L^2(\mathbb{R})$
satisfies
\begin{equation}
\label{bound-before}
\|u_0-u_{\lambda_0}(\cdot,0)\|_{L^2}+\|v_0-v_{\lambda_0}(\cdot,0)\|_{L^2} \leq \epsilon,
\end{equation}
then for every $t\in \mathbb{R}$, there exists $\lambda\in \mathbb{C}$ such that
\begin{equation}
\label{bound-parameters}
|\lambda-\lambda_0| \leq C \epsilon
\end{equation}
and
\begin{equation}
\label{bound-after}
\inf_{a,\theta\in \mathbb{R}}(\|u(\cdot+a,t)-e^{-i\theta}u_{\lambda}(\cdot,t)\|_{L^2}
+\|v(\cdot+a,t)-e^{-i\theta}v_{\lambda}(\cdot,t)\|_{L^2}) \leq C \epsilon,
\end{equation}
where the constant $C$ is independent of $\epsilon$ and $t$.
\end{thm}

The proof of Theorem \ref{maintheorem} relies on the auto-B\"{a}cklund transformation of the MTM system
and perturbation analysis. Our approach follows the strategy used by Mizumachi and Pelinovsky in \cite{PM}
for  proving $L^2$-orbital stability of the NLS solitons. Their result was extended by
Contreras and Pelinovsky in \cite{CoP} to multi-solitons of the NLS equation by using a more general dressing
transformation. Furthermore, the recent work \cite{CP} of Cuccagna and Pelinovsky shows how
an asymptotic stability of the NLS solitons can be deduced by combining the auto-B\"{a}cklund transformation
and the nonlinear steepest descent method.

A similar type of ideas have been applied to other completely integrable systems.
We mention  for example the work \cite{MV} of  Merle and Vega where they prove $L^2$-stability
and asymptotic stability of the KdV solitons by using of the Miura transformation that
relates the KdV solitons to the kinks of the defocusing modified KdV equation. The work \cite{MP} of Mizumachi and Pego makes
use of a linearized B\"{a}cklund transformation to establish an asymptotic stability of Toda lattice
solitons. Hoffman and Wayne \cite{HW} formulated an abstract orbital stability result for
soliton solutions of integrable equations that can be achieved via B\"{a}cklund transformations.

In addition to an increasing popularity of the integrability techniques to study nonlinear
stability of soliton solutions, we note that such techniques become
particularly powerful for the MTM system.
Compared to the NLS equation, proofs of global existence and orbital stability of solitons in
the nonlinear Dirac equations (including the MTM system) are complicated by the fact that
the quadratic part of the Hamiltonian of the nonlinear Dirac equations is not bounded from neither above
nor below.  Consequently, there exist two bands of continuous spectrum of the linearized Dirac operator
for positive and negative energies, which extend to positive and negative infinities. For this reason,
proof of orbital stability of Dirac solitons poses a serious difficulty to the application of  standard energy arguments
 \cite{GSS}. There are, nevertheless, many works that deal with spectral
properties of linearized Dirac operators  \cite{BerCom, BouCom, ChPel-cme, ComPrep1, ComPrep2, ComPrep3, DerGot, KL1}.
Also, asymptotic stability of small solitary waves in the general nonlinear Dirac equations
has been considered \cite{KomKom, Kop, PS} (see \cite{Bou2006, Bou2008, BouCuc} for similar results in
the space of three dimensions).

Other than these works, not much is known about the orbital stability of Dirac solitons.
The recent work \cite{PS2} of Pelinovsky and Shimabukuro incorporates the inverse scattering method
to obtain an additional conserved quantity that can serve as a Lyapunov functional in the proof of
$H^1$-orbital stability of the MTM solitons. We continue to use the inverse scattering method in
this paper and prove $L^2$-orbital stability of the MTM solitons by a non-variational method
that relies on the auto-B\"{a}cklund transformation.

B\"{a}cklund transformations are used to generate solutions of a differential system, usually depending on a parameter,
from another solution of another differential system. When this transformation relates two different solutions of the same system,
it is called the auto-B\"{a}cklund transformation.  These transformations, when found, can be used to link the
orbital stability of a certain class of solutions to that of another class of solutions. In particular,
a stable neighborhood of the zero solution can be mapped to a stable neighborhood of one-soliton solution,
and vice verse. However, there is no systematic way to find such transformations and, to the best of our knowledge,
this is the first appearance of the auto-B\"{a}cklund transformation for the MTM system (\ref{MTM})
in the literature.

We note in passing that the inverse scattering method of the associated Kaup-Newell system \cite{KN}
has been extensively studied and the auto-B\"{a}cklund transformation of other
related integrable equations have been reported in the literature. In particular, the work \cite{Imai} of Imai
reports the Darboux transformation of the derivative nonlinear Schr\"{o}dinger equation
and claims that the Darboux transformation of the MTM system can be obtained similarly,
but no details are given. Furthermore, the Coleman correspondence between the MTM system and the sine-Gordon equation
is studied through the inverse scattering transform \cite{KN,O}
and this may also yield another derivation of the auto--B\"{a}cklund transformation for the MTM system
because the auto--B\"{a}cklund transformation of the sine--Gordon equation is well known. For our purposes,
we derive the auto--B\"{a}cklund transformation of the MTM system by using Ricatti equations and symbolic computations.

The paper is organized as follows. Section 2 introduces the auto-B\"{a}cklund transformation for the MTM system
(\ref{MTM}) and uses it to recover the MTM solitons (\ref{onesol}) from zero solutions.
We also list the Lorentz transformation
for the MTM system (\ref{MTM}) and outline the steps in the proof of Theorem \ref{maintheorem}. Section 3
reports details of the transformation from perturbed one-soliton solutions to small solutions
at the initial time $t = 0$. Section 4 describes the transformation from small solutions
to perturbed one-soliton solutions for all times $t \in \mathbb{R}$.
Section 5 completes the proof of Theorem \ref{maintheorem}.

\section{B\"{a}cklund transformation for the MTM system}

We begin by introducing the Lax pair and the auto-B\"{a}cklund transformation of the MTM system (\ref{MTM})
in the laboratory coordinates. Then, we give the Lorentz transformation for the MTM system
(\ref{MTM}) and outline the steps in the proof of Theorem \ref{maintheorem}.

The Lax pair of the MTM system (\ref{MTM}) is defined in terms of the following two linear operators:
\begin{equation} \label{lablax1}
L=\frac{i}{4}(|u|^2-|v|^2)\sigma_3-\frac{i\lambda}{2}\left(\begin{matrix} 0 & \overline{v} \\ v & 0 \end{matrix}\right)
+ \frac{i}{2\lambda} \left(\begin{matrix} 0 & \overline{u} \\ u & 0 \end{matrix}\right) +\frac{i}{4}\left(\lambda^2-\frac{1}{\lambda^2}\right)\sigma_3
\end{equation}
and
\begin{equation} \label{lablax2}
A=-\frac{i}{4}(|u|^2+|v|^2)\sigma_3-\frac{i\lambda}{2}\left(\begin{matrix} 0 & \overline{v} \\ v & 0 \end{matrix}\right) - \frac{i}{2\lambda} \left(\begin{matrix} 0 & \overline{u} \\ u & 0 \end{matrix}\right) +\frac{i}{4}\left(\lambda^2+\frac{1}{\lambda^2}\right)\sigma_3.
\end{equation}
The formal compatibility condition $\vec{\phi}_{xt}=\vec{\phi}_{tx}$ for the system of linear equations
\begin{equation} \label{laxeq}
\vec{\phi}_x=L\vec{\phi}\quad  \mbox{and}\quad \vec{\phi}_t=A\vec{\phi}
\end{equation}
yields the MTM system (\ref{MTM}).

Note that the solution $(u,v)$ of the MTM system (\ref{MTM}) appears as coefficients of differential equations
in the linear system \eqref{laxeq}. The auto-B\"{a}cklund transformation relates two solutions of
the MTM system (\ref{MTM}) while preserving the linear system \eqref{laxeq}.
Now let us state the auto-B\"{a}cklund transformation.

\begin{prop} \label{backlund}
Let $(u,v)$ be a $C^1$ solution of the MTM system (\ref{MTM}) and
$\vec{\phi}=(\phi_1,\phi_2)^t$ be a $C^2$ nonzero solution of the linear
system \eqref{laxeq} associated with the potential $(u,v)$ and the spectral
parameter $\lambda = \delta e^{i \gamma/2}$. Then, the following transformation
\begin{equation}\label{backu}
\mathbf{u} = -u \frac{e^{-i\gamma/2}|\phi_1|^2+
e^{i\gamma/2}|\phi_2|^2}{e^{i\gamma/2}|\phi_1|^2+e^{-i\gamma/2}|\phi_2|^2}+
\frac{2i \delta^{-1} \sin\gamma \overline{\phi}_1\phi_2}{e^{i\gamma/2}|\phi_1|^2+e^{-i\gamma/2}|\phi_2|^2}
\end{equation}
and
\begin{equation}\label{backv}
\mathbf{v} = -v \frac{e^{i\gamma/2}|\phi_1|^2+
e^{-i\gamma/2}|\phi_2|^2}{e^{-i\gamma/2}|\phi_1|^2+e^{i\gamma/2}|\phi_2|^2}-
\frac{2i\delta \sin\gamma \overline{\phi}_1\phi_2}{e^{-i\gamma/2}|\phi_1|^2+e^{i\gamma/2}|\phi_2|^2},
\end{equation}
generates a new $C^1$ solution of the MTM system (\ref{MTM}).
Furthermore, the transformation
\begin{equation} \label{defpsi}
\psi_1=\frac{\overline{\phi}_2}{|e^{i\gamma/2}|\phi_1|^2+e^{-i\gamma/2}|\phi_2|^2|}, \quad \psi_2=\frac{\overline{\phi}_1}{|e^{i\gamma/2}|\phi_1|^2+e^{-i\gamma/2}|\phi_2|^2|}
\end{equation}
yields a new $C^2$ nonzero solution $\vec{\psi} = (\psi_1, \psi_2)^t$ of the linear system \eqref{laxeq}
associated with the new potential $(\mathbf{u},\mathbf{v})$ and the same spectral parameter $\lambda$.
\end{prop}

\begin{proof}
Setting $\Gamma = \phi_1/\phi_2$ in the linear system (\ref{laxeq}) with Lax operators
\eqref{lablax1} and \eqref{lablax2} yields the Riccati equations
\begin{equation} \label{xriccati}
\left\{ \begin{array}{l}
\Gamma_x=2i(\rho_2^2-\rho_1^2)\Gamma+\frac{i}{2}(|u|^2-|v|^2)\Gamma + i(\rho_2v-\rho_1u)\Gamma^2-i(\rho_2\overline{v}-\rho_1\overline{u}),\\
\Gamma_t=2i(\rho_2^2+\rho_1^2)\Gamma-\frac{i}{2}(|u|^2+|v|^2)\Gamma + i(\rho_2v+\rho_1u)\Gamma^2-i(\rho_2\overline{v}+\rho_1\overline{u}),
\end{array} \right.
\end{equation}
where $\rho_1=\frac{1}{2 \lambda}$ and $\rho_2=\frac{\lambda}{2}$. If we choose $\Gamma' :=\frac{1}{\overline{\Gamma}}$,
$\mathbf{u} := M(\Gamma; \rho_1)f(\Gamma;u,\rho_1)$, and $\mathbf{v} := M(\Gamma;\rho_2)f(\Gamma; v,\rho_2)$ with
$$
M(\Gamma; k) = -\frac{k|\Gamma|^2+\overline{k}}{\overline{k}|\Gamma|^2+k}, \quad
f(\Gamma;q,k) = q + \frac{4i \mbox{Im}(k^2) \overline{\Gamma}}{k|\Gamma|^2+\overline{k}},
$$
then the Riccati equations \eqref{xriccati} remain invariant in variables $\Gamma'$, $\mathbf{u}$, and $\mathbf{v}$.
The transformation formulas above yield representation \eqref{backu} and \eqref{backv}.
Note that if $\vec{\phi} = \vec{0}$ at one point $(x_0,t_0)$, then $\vec{\phi} = \vec{0}$ for all $(x,t)$.
If $(u,v)$ is $C^1$ in $(x,t)$, $\vec{\phi}$ is $C^2$ in $(x,t)$, and $\vec{\phi} \neq \vec{0}$, then
$(\mathbf{u},\mathbf{v})$ is $C^1$ for every $x \in \mathbb{R}$ and $t \in \mathbb{R}$.

The validity of the transformation (\ref{defpsi}) has been verified with Wolfram's Mathematica. Again, if
$\vec{\phi}$ is $C^2$ in $(x,t)$ and $\vec{\phi} \neq \vec{0}$, then
$\vec{\psi}$ is $C^2$ and $\vec{\psi} \neq \vec{0}$ for every $x \in \mathbb{R}$ and $t \in \mathbb{R}$.
\end{proof}

Let us denote the transformations \eqref{backu}--\eqref{backv} by $\mathcal{B}$, hence
$$
\mathcal{B}: (u,v,\vec{\phi},\lambda)\mapsto (\mathbf{u},\mathbf{v}),
$$
where $\vec{\phi}$ is a corresponding vector of the linear system \eqref{laxeq}
associated with the potential $(u,v)$ and the spectral parameter $\lambda$.

In the simplest example, the MTM soliton (\ref{onesol}) is recovered by the
transformations \eqref{backu} and \eqref{backv} from the zero solution $(u,v)=(0,0)$, that is,
$$
\mathcal{B}: (0,0,\vec{\phi},\lambda)\mapsto (u_{\lambda},v_{\lambda}).
$$
Indeed, a solution satisfying the linear system \eqref{laxeq} with $(u,v)=(0,0)$ is given by
\begin{equation}
\label{phi-sol}
\left\{ \begin{array}{l}
\phi_1(x,t) = e^{\frac{i}{4}(\lambda^2-\lambda^{-2})x+\frac{i}{4}(\lambda^2+\lambda^{-2})t}, \\
\phi_2(x,t) = e^{-\frac{i}{4}(\lambda^2-\lambda^{-2})x-\frac{i}{4}(\lambda^2+\lambda^{-2})t}. \end{array} \right.
\end{equation}
Substituting this expression into \eqref{backu} and \eqref{backv} yields $(\mathbf{u},\mathbf{v})=(u_{\lambda},v_{\lambda})$
given by (\ref{onesol}).

Another important example is a transformation from the MTM solitons (\ref{onesol}) to the zero solution.
We shall only give the explicit expressions of this transformation for the case $|\lambda| = \delta = 1$.
By \eqref{defpsi} and \eqref{phi-sol}, we can find the vector $\vec{\psi}$ solving the linear system
\eqref{laxeq} with $(u_{\lambda},v_{\lambda})$ given by (\ref{onesol}).
When $\lambda = e^{i \gamma/2}$, the vector $\vec{\psi}$ is given by
\begin{equation}
\label{psi-sol}
\left\{ \begin{array}{l}
\psi_1(x,t) = e^{\frac{1}{2} x \sin\gamma + \frac{i}{2} t \cos \gamma} \left| \sech\left(x \sin\gamma - i\frac{\gamma}{2}\right) \right|, \\
\psi_2(x,t) = e^{-\frac{1}{2} x \sin\gamma - \frac{i}{2} t \cos \gamma} \left| \sech\left(x \sin\gamma - i\frac{\gamma}{2}\right) \right|.
\end{array} \right.
\end{equation}
We note that $\vec{\psi}$ has exponential decay as $|x| \to \infty$ and, therefore, it is an eigenvector
of the linear system \eqref{laxeq} for the eigenvalue $\lambda = e^{i \gamma/2}$.
Substituting the eigenvector $\vec{\psi}$ into the transformation
(\ref{backu}) and (\ref{backv}), we obtain the zero solution
from the MTM soliton, that is,
$$
\mathcal{B}: (u_{\lambda},v_{\lambda},\vec{\psi},\lambda)\mapsto (0,0).
$$

When $|\lambda| = \delta = 1$ for $(u_{\lambda},v_{\lambda})$ given by \eqref{onesol}, we realize that $\nu=0$
and hence the MTM solitons (\ref{onesol}) are stationary. Travelling MTM solitons with $\nu \neq 0$
can be recovered from the stationary MTM solitons with $\nu = 0$ by the Lorentz transformation.
Hence, without loss of generality, we can choose $\lambda_0 = e^{i\gamma_0/2}$ for a fixed
$\gamma_0 \in(0,\pi)$ in Theorem \ref{maintheorem}. Let us state the Lorentz transformation,
which can be verified with the direct substitutions.

\begin{prop}
Let $(u,v)$ be a solution of the MTM system (\ref{MTM}) and
let $\vec{\phi}$ be a solution of the linear system (\ref{laxeq})
associated with $(u,v)$ and $\lambda = e^{i\gamma/2}$. Then,
\begin{eqnarray}\label{lorentz}
\left\{ \begin{array}{l}
u'(x,t) := \delta^{-1} u(k_1 x + k_2 t, k_1 t + k_2 x),\\
v'(x,t) := \delta v(k_1 x + k_2 t, k_1 t + k_2 x),\end{array}\right. \quad
k_1 := \frac{\delta^2 + \delta^{-2}}{2}, \quad k_2 := \frac{\delta^2 - \delta^{-2}}{2},
 \end{eqnarray}
is a new solution of the same MTM system (\ref{MTM}), whereas
\begin{equation}
\vec{\phi}'(x,t) := \vec{\phi}(k_1 x + k_2 t, k_1 t + k_2 x),
\end{equation}
is a new solution of the linear system (\ref{laxeq}) associated with $(u',v')$
and $\lambda = \delta e^{i\gamma/2}$. \label{prop-Lorentz}
\end{prop}

The stationary MTM solitons at $t = 0$ can be written by using the expressions
\begin{equation}\label{Soliton}
\left\{ \begin{array}{l}
u_{\gamma}(x) = i\sin \gamma \sech\left(x\sin\gamma - i \frac{\gamma}{2} \right), \\
v_{\gamma}(x) = -i\sin\gamma \sech\left(x\sin\gamma + i \frac{\gamma}{2} \right),\end{array} \right.
\end{equation}
that depend on the parameter $\gamma \in (0,\pi)$.
The time oscillation, gauge translation, and space translation can be included
with the help of the transformation
\begin{equation}\label{Soliton-Transformation}
\left\{ \begin{array}{l}
u(x,t) = e^{i\theta-it \cos\gamma} u_{\gamma}(x+a), \\
v(x,t) = e^{i\theta-it \cos\gamma} v_{\gamma}(x+a),\end{array} \right.
\end{equation}
where $\theta,a \in \mathbb{R}$ are two translational parameters of the stationary MTM solitons.

Let us now describe the ideas of our method for the proof of Theorem \ref{maintheorem}.
First we clarify some notations: $(u_{\gamma_0},v_{\gamma_0})$
denotes one-soliton solution given by \eqref{Soliton} with a fixed $\gamma_0\in(0,\pi)$,
$\vec{\psi}_{\gamma_0}$ denotes the corresponding eigenvector given by (\ref{psi-sol}) for $t = 0$,
whereas $L(u,v,\lambda)$ and $A(u,v,\lambda)$ denote the Lax operators $L$ and $A$ that contain
$(u,v)$ and a spectral parameter $\lambda$.

The main steps for the proof of Theorem \ref{maintheorem} are the following.
First, we fix an initial data $(u_0,v_0)\in H^2(\mathbb{R})$ such that $(u_0,v_0)$ is
sufficiently close to $(u_{\gamma_0},v_{\gamma_0})$ in $L^2$-norm,
according to the bound (\ref{bound-before}).

\vspace{0.2cm}

\emph{Step 1: From a perturbed one-soliton solution to a small solution at $t=0$.}
In this step, we need to study the vector solution $\vec{\psi}$ of
the linear equation
\begin{equation}
\label{L-introduction}
\partial_x\vec{\psi}=L(u_0,v_0,\lambda)\vec{\psi} \quad \mbox{\rm at time} \; t=0.
\end{equation}
In addition to proving the existence of an exponentially decaying solution $\vec{\psi}$ of the linear equation
(\ref{L-introduction}) for an eigenvalue $\lambda$, we need to prove that if $(u_0,v_0)$ is close to
$(u_{\gamma_0},v_{\gamma_0})$ in $L^2$-norm, then $\vec{\psi}$ is close to $\vec{\psi}_{\gamma_0}$ in $H^1$-norm
and $\lambda$ is close to $e^{i\gamma_0/2}$. Parameter $\lambda$ in bound (\ref{bound-parameters}) is now determined by
the eigenvalue of the linear equation (\ref{L-introduction}).

The earlier example of obtaining the zero solution from the one-soliton solution gives
a good insight that the auto-B\"{a}cklund transformation given by
Proposition \ref{backlund} produces a function  $(p_0,q_0)$ in an $L^2$-neighborhood of the zero solution at $t=0$,
\begin{equation}\label{pq0}
\mathcal{B} : (u_0,v_0,\vec{\psi},\lambda) \mapsto (p_0,q_0).
\end{equation}

\vspace{0.2cm}

\emph{Step 2: Time evolution of the transformed solution.}
Since $(p_0,q_0)$ is close to the zero solution in $L^2$-norm,
then so is $(p(\cdot,t),q(\cdot,t))$ for all $t \in \mathbb{R}$,
thanks to the $L^2$-conservation (\ref{l2-conservation}) of the MTM system (\ref{MTM})
in Theorem \ref{theorem-Candy}.

\vspace{0.2cm}

\emph{Step 3: From a small solution to a perturbed one-soliton solution for all times $t\in \mathbb{R}$.}
In this step, we are interested in the existence problem of the vector function $\vec{\phi}$ that solves the linear system
\begin{equation}
\label{LL}
\partial_x\vec{\phi}=L(p,q,\lambda)\vec{\phi}, \quad
\partial_t\vec{\phi}=A(p,q,\lambda)\vec{\phi}
\end{equation}
where $(p,q)$ is the solution of the MTM system (\ref{MTM}) starting with
the initial data $(p,q) |_{t = 0} = (p_0,q_0)$ in $H^2(\mathbb{R})$.
Using the vector $\vec{\phi}$ and the auto-B\"{a}cklund transformation given by Proposition \ref{backlund},
we obtain a solution $(u,v)$ of the MTM system (\ref{MTM})
in an $L^2$-neighborhood of one-soliton solution:
\begin{equation}
\label{BB}
\mathcal{B} : (p,q,\vec{\phi},\lambda) \mapsto (u,v).
\end{equation}
Some translational parameter $a$ and $\theta$ arise in the construction
of the most general solution of the linear system (\ref{LL}). Bound (\ref{bound-after})
is found from the analysis of the auto--B\"{a}cklund transformation (\ref{BB}).

\vspace{0.2cm}

To summarize here, there are three key ingredients along the way: mapping an $L^2$-neighborhood of
one-soliton solution to that of the zero solution at $t = 0$, the $L^2$-conservation of the MTM system,
and mapping an $L^2$-neighborhood of the zero solution to that of the one-soliton solution for every $t \in \mathbb{R}$.
As a result, if the initial data is sufficiently close to the one-soliton solution in
$L^2$ as in the initial bound (\ref{bound-before}), then the solution of the MTM system remains
close to the one-soliton solution in $L^2$ for all times as in the final bound (\ref{bound-after}).
A schematic picture is as follows:

{\centering
 \begin{tikzpicture}[descr/.style={fill=white, inner sep=2.5pt}]
 \matrix (m) [matrix of math nodes, row sep=3em, column sep=3em]
 {(u_0,v_0) & & (u,v)\\
 (p_0,q_0) & & (p,q) \\};
 \path[->, font=\scriptsize]
 (m-1-1) edge node[auto] {$\mathcal{B}$} (m-2-1)
 (m-2-1) edge (m-2-3)
  (m-2-3) edge node[auto] {$\mathcal{B}$} (m-1-3);
  \path[dotted] (m-1-1) edge (m-1-3);
\end{tikzpicture}
\par}
From a technical point, perturbation theory and the method of Lyapunov-Schmidt
reductions are used to obtain a mapping from the neighborhood of one-soliton
solutions to that of the zero solution, and a standard fixed-point argument
yields the  passage from the zero solution to one-soliton solutions.

Finally, we can remove the technical assumption that $(u_0,v_0) \in H^2(\mathbb{R})$ by
an approximation argument in $L^2(\mathbb{R})$ thanks to the fact that the bounds
(\ref{bound-parameters}) and (\ref{bound-after}) are found to be uniform for the
approximating sequence associated with the initial data $(u_0,v_0) \in L^2(\mathbb{R})$.

We note that the solution $(p,q)$ of the MTM system (\ref{MTM})
in a $L^2$-neighborhood of the zero solution could contain
some $L^2$-small MTM solitons, which are related to the discrete spectrum
of the linear system (\ref{LL}).
Sufficient conditions for the absence of the discrete spectrum were derived
in \cite{Pel-survey}, and the $L^2$ smallness of the initial data is not
generally sufficient for excluding eigenvalues of the discrete spectrum. If the small solitons
occur in the Cauchy problem, asymptotic decay of solutions to the MTM solitons
(\ref{onesol}) can not be proved, in other words, $(p,q)$ do not decay to $(0,0)$ in $L^{\infty}$-norm
as $t \to \infty$. Therefore, a more restrictive hypothesis on the initial data is
generally needed to establish asymptotic stability of MTM solitons.
See \cite{CP} for restrictions on initial data required in the proof of
asymptotic stability of NLS solitons.

We also note that modulation equations for parameters $a$ and $\theta$ in
Theorem \ref{maintheorem} are not included in this approach. This can be viewed
as an advantage of the auto-B\"{a}cklund transformation,which does not rely on the global
control of the dynamics of $a$ and $\theta$ in the modulation equations.
Values of $a$ and $\theta$ are related to arbitrary constants that appear in the
construction of $\vec{\phi}$ as a solution of the linear system (\ref{LL}).
These values are eliminated in the infimum norm stated in the orbital stability result
(\ref{bound-after}) in Theorem \ref{maintheorem}.

\section{From a perturbed one-soliton solution to a small solution}

Here we use the auto-B\"{a}cklund transformation given by Proposition \ref{backlund} to transform
a $L^2$-neighborhood of one-soliton solution to that of the zero solution at $t=0$.
Let $(u_0,v_0) \in L^2(\mathbb{R})$ be the initial data of the MTM system (\ref{MTM})
satisfying bound (\ref{bound-before}) for $\lambda_0 = e^{i \gamma_0/2}$.
Let $\vec{\psi}$ be a decaying eigenfunction of the spectral problem
\begin{equation} \label{laxeq2}
\partial_x \vec{\psi} = L(u_0,v_0,\lambda) \vec{\psi},
\end{equation}
for an eigenvalue $\lambda$. First, we show that under the condition (\ref{bound-before}),
an eigenvector $\vec{\psi}$ always exists and $\lambda$ is close to $\lambda_0$.
Then, we write $\lambda = \delta e^{i \gamma/2}$ and define
\begin{equation}\label{up}
p_0 := -u_0 \frac{e^{-i\gamma/2}|\psi_1|^2+e^{i\gamma/2}|\psi_2|^2}{e^{i\gamma/2}|\psi_1|^2+e^{-i\gamma/2}|\psi_2|^2} +
\frac{2i \delta^{-1} \sin\gamma \overline{\psi}_1\psi_2}{e^{i\gamma/2}|\psi_1|^2+e^{-i\gamma/2}|\psi_2|^2}
\end{equation}
and
\begin{equation}\label{vq}
q_0 := -v_0 \frac{e^{i\gamma/2}|\psi_1|^2+e^{-i\gamma/2}|\psi_2|^2}{e^{-i\gamma/2}|\psi_1|^2+e^{i\gamma/2}|\psi_2|^2}-
\frac{2i \delta \sin\gamma \overline{\psi}_1\psi_2}{e^{-i\gamma/2}|\psi_1|^2+e^{i\gamma/2}|\psi_2|^2}.
\end{equation}
We intend to show that $(p_0,q_0)$ is small in $L^2$ norm.

When $(u_0,v_0) = (u_{\gamma_0},v_{\gamma_0})$, the spectral problem \eqref{laxeq2}
has exactly one decaying eigenvector $\vec{\psi}$ given by
\begin{equation}
\label{psi-sol-again}
\left\{ \begin{array}{l}
\psi_1(x) = e^{\frac{1}{2} x \sin\gamma_0} \left| \sech\left(x \sin\gamma_0 - i\frac{\gamma_0}{2}\right) \right|, \\
\psi_2(x) = e^{-\frac{1}{2} x \sin\gamma_0} \left| \sech\left(x \sin\gamma_0 - i\frac{\gamma_0}{2}\right) \right|,
\end{array} \right.
\end{equation}
which corresponds to the eigenvalue $\lambda = \lambda_0 = e^{i \gamma_0/2}$.
The other linearly independent solution $\vec{\xi}$ of the spectral problem \eqref{laxeq2}
with $\lambda = \lambda_0$ is given by
\begin{eqnarray}
\label{eigfcn}
\left\{ \begin{array}{l}
\xi_1(x) = e^{\frac{1}{2} x \sin\gamma_0}(e^{-2x \sin\gamma_0}-x\sin(2\gamma_0)) |\sech\left(x \sin\gamma_0 -i\frac{\gamma_0}{2}\right)|, \\
\xi_2(x) = -e^{-\frac{1}{2} x \sin\gamma_0}(e^{2x\sin\gamma_0}+2\cos\gamma_0 + x\sin(2\gamma_0)) |\sech\left(x \sin\gamma_0 - i\frac{\gamma_0}{2}\right)|.
\end{array} \right.
\end{eqnarray}
This solution grows exponentially as $|x| \to \infty$.
Therefore, $\dim\ker(\partial_x-L(u_{\gamma_0},v_{\gamma_0},\lambda_0))=1$. For clarity,
we denote the decaying eigenvector (\ref{psi-sol-again}) by $\vec{\psi}_{\gamma_0}$.

When $(u_0,v_0)$ is close to $(u_{\gamma_0},v_{\gamma_0})$ in $L^2$-norm,
we would like to construct a decaying solution $\vec{\psi}$ of the spectral problem \eqref{laxeq2},
which is close to the eigenvector $\vec{\psi}_{\gamma_0}$. This is achieved in Lemma \ref{onebound1} below.
To simplify analysis, we introduce a unitary transformation in the linear equation \eqref{laxeq2},
\begin{equation}\label{unitary}
\vec{\psi} = \left[\begin{matrix} f & 0\\0&\overline{f} \end{matrix}\right]\vec{\phi},
\end{equation}
where $f(x)=e^{\frac{i}{4}\int_{0}^x(|u_0|^2-|v_0|^2)dx}$ is well defined for any
$(u_0,v_0) \in L^2(\mathbb{R})$.  Then, the linear equation \eqref{laxeq2}  becomes
\begin{equation}\label{sys}
\partial_x\vec{\phi} = M(u_0,v_0,\lambda) \vec{\phi},
\end{equation}
where
$$
M(u_0,v_0,\lambda) := \frac{i}{4} \left[\begin{matrix} \lambda^2 - \lambda^{-2} &
2 (\overline{u}_0 \lambda^{-1} - \overline{v}_0 \lambda)\overline{f}^2 \\
2 (u_0 \lambda^{-1} - v_0 \lambda)f^2 & \lambda^{-2} - \lambda^2 \end{matrix}\right].
$$
The following lemma gives the main result of the perturbation theory.
Below, $A\lesssim B$ means that there exists a positive $\epsilon$-independent constant $C$
such that $A\leq CB$ for all sufficiently small $\epsilon$.

\begin{lem}\label{onebound1}
For a fixed $\lambda_0 = e^{i\gamma_0/2}$ with $\gamma_0 \in (0,\pi)$,
there exist a real positive $\epsilon$ such that if
\begin{equation}
\label{bound-eigenvalue-before}
\|u_0-u_{\gamma_0}\|_{L^2}+\|v_0-v_{\gamma_0}\|_{L^2}\leq \epsilon,
\end{equation}
then there exists a solution $\vec{\psi} \in H^1(\mathbb{R};\mathbb{C}^2)$
of the linear equation  \eqref{laxeq2} for an eigenvalue $\lambda \in \mathbb{C}$ such that
\begin{equation}
\label{bound-eigenvalue}
|\lambda-\lambda_0|+\|\vec{\psi}-\vec{\psi}_{\gamma_0}\|_{H^1} \lesssim
\|u_0-u_{\gamma_0}\|_{L^2}+\|v_0-v_{\gamma_0}\|_{L^2}.
\end{equation}
\end{lem}

\begin{proof}
We divide the proof into four steps that accomplish the method of Lyapunov--Schmidt reductions.
Step 1 is a set-up for the perturbation theory. Step 2 splits the problem into two parts by appropriate projections.
In Step 3, we solve the first part of the problem by using the implicit function theorem.
In Step 4, we solve the residual equation that determines uniquely $\lambda \in \mathbb{C}$ and
$\vec{\psi} \in H^1(\mathbb{R};\mathbb{C}^2)$ satisfying bound (\ref{bound-eigenvalue}).

\vspace{0.2cm}

\emph{Step 1.}
Set $u_0 = u_{\gamma_0}+u_s$ and $v_0 = v_{\gamma_0}+v_s$, where
$(u_s, v_s) \in L^2(\mathbb{R})$ are remainder terms, which are $\mathcal{O}(\epsilon)$ small in $L^2$ norm,
according to the bound (\ref{bound-eigenvalue-before}).
We expand $1/\lambda^2$ and $1/\lambda$ around $\lambda_0$ by Taylor series.
Using the fact $|u_{\gamma_0}| = |v_{\gamma_0}|$, we also expand $u_0 f^2$ and $v_0 f^2$
in Taylor series, e.g.
\begin{eqnarray}
\label{Taylor-series-expansion}
u_0 f^2 = u_0 e^{\frac{i}{2}\int_0^x(|u_0|^2-|v_0|^2)dx} = (u_{\gamma_0}+u_s) \left( 1+g+\frac{1}{2}g^2+\mathcal{O}(g^3)\right),
\end{eqnarray}
where
$$
g := i \int_0^x \mbox{Re} \left( u_s\overline{u}_{\gamma_0}-v_s\overline{v}_{\gamma_0} \right) dx
+ \frac{i}{2} \int_0^x(|u_s|^2-|v_s|^2) dx.
$$
Note that $g$ is well defined for $(u_s, v_s) \in L^2(\mathbb{R})$. From these expansions,
the linear equation \eqref{sys} becomes
\begin{equation}\label{deltaeq}
(\partial_x - M_{\gamma_0})\vec{\phi}= \Delta M \vec{\phi},
\end{equation}
where
$$
M_{\gamma_0} = M(u_{\gamma_0},v_{\gamma_0},\lambda_0) = \frac{1}{2} \left[\begin{matrix} -\sin\gamma_0 &
i(e^{-i\gamma_0/2} \overline{u}_{\gamma_0}-e^{i\gamma_0/2}\overline{v}_{\gamma_0} ) \\
i(e^{-i\gamma_0/2} u_{\gamma_0}-e^{i\gamma_0/2} v_{\gamma_0} ) & \sin\gamma_0 \end{matrix}\right]
$$
and the perturbation term $\Delta M$ applied to any $\vec{\phi} \in H^1(\mathbb{R})$
satisfies the inequality
\begin{equation}\label{Lbound}
 \| \Delta M \vec{\phi} \|_{L^2} \lesssim (|\lambda-\lambda_0|+\|u_s\|_{L^2}+\|v_s\|_{L^2}) \| \phi \|_{H^1},
\end{equation}
thanks to the embedding of $H^1(\mathbb{R})$ in $L^{\infty}(\mathbb{R}) \cap L^2(\mathbb{R})$.
Note that the bound (\ref{Lbound}) can not be derived in the context of the spectral problem (\ref{laxeq2})
without the unitary transformation \eqref{unitary}, which removes the term $\frac{i}{4} (|u|^2 - |v|^2) \sigma_3$
from the operator $L$ in (\ref{lablax1}). This explains a posteriori why we are using
the technical transformation (\ref{unitary}).

We will later need the explicit computation of the leading order part in the perturbation
term $\Delta M$ with respect to $(\lambda - \lambda_0)$, that is,
\begin{equation}
\label{explicit-M}
\Delta M = \frac{i}{2} (\lambda-\lambda_0) \left[\begin{matrix}(\lambda_0 + \lambda_0^{-3}) &
-(\overline{u}_{\gamma_0} \lambda_0^{-2} + \overline{v}_{\gamma_0} ) \\ -(u_{\gamma_0} \lambda_0^{-2} +v_{\gamma_0} ) &
-(\lambda_0+\lambda_0^{-3}) \end{matrix}\right] + \mathcal{O}((\lambda-\lambda_0)^2,\|u_s \|_{L^2},\|v_s\|_{L^2}).
\end{equation}

\vspace{0.2cm}

\emph{Step 2.}  We aim to construct an appropriate projection operator by which
we split the linear equation \eqref{deltaeq} into two parts.
Recall that $\mbox{\rm dim} \ker(\partial_x - M_{\gamma_0}) = 1$ and
let $\vec{\phi}_{\gamma_0} \in \ker(\partial_x - M_{\gamma_0})$ and
$\vec{\eta}_{\gamma_0} \in \ker(\partial_x + M_{\gamma_0}^*)$. These null
vectors can be obtained explicitly:
$$
\vec{\phi}_{\gamma_0}(x) = \left[ \begin{matrix}e^{\frac{x}{2}\sin\gamma_0} \\ e^{-\frac{x}{2}\sin\gamma_0} \end{matrix}\right]
\left| \sech\left(\sin\gamma_0 x-i\frac{\gamma_0}{2}\right) \right|, \quad
\vec{\eta}_{\gamma_0}(x) = \left[ \begin{matrix}e^{-\frac{x}{2}\sin\gamma_0} \\-e^{\frac{x}{2}\sin\gamma_0} \end{matrix}\right]
\left| \sech\left(\sin\gamma_0 x-i\frac{\gamma_0}{2}\right) \right|.
$$
We note that
$\langle \vec{\eta}_{\gamma_0},\vec{\phi}_{\gamma_0}\rangle_{L^2}=0$
but $\langle \sigma_3 \vec{\eta}_{\gamma_0},\vec{\phi}_{\gamma_0}\rangle_{L^2} \neq 0$,
where $\sigma_3=\left[\begin{matrix}1&0\\0&-1\end{matrix}\right]$.
Also note that $\vec{\phi}_{\gamma_0} = \vec{\psi}_{\gamma_0}$ given by (\ref{psi-sol-again})
because $|u_{\gamma_0}| = |v_{\gamma_0}|$.

Let us make the following decomposition:
\begin{equation}\label{decomp}
\vec{\phi}=\vec{\phi}_{\gamma_0}+\vec{\phi}_s,
\end{equation}
where $\vec{\phi}_s$ is defined uniquely from the normalization condition
$\langle \sigma_3 \vec{\eta}_{\gamma_0},\vec{\phi} \rangle_{L^2} =
\langle \sigma_3 \vec{\eta}_{\gamma_0},\vec{\phi}_{\gamma_0}\rangle_{L^2}$,
which yields the orthogonality condition $\langle \sigma_3 \vec{\eta}_{\gamma_0},\vec{\phi}_{s}\rangle_{L^2} = 0$.
Then we introduce the projection operator $P_{\gamma_0} : L^2(\mathbb{R};\mathbb{C}^2) \rightarrow
L^2(\mathbb{R};\mathbb{C}^2)\cap \Span\{\sigma_3 \vec{\eta}_{\gamma_0}\}^{\perp}$ defined by
$$
P_{\gamma_0} \vec{\phi}=\vec{\phi}-\frac{\langle \sigma_3 \vec{\eta}_{\gamma_0},\vec{\phi}\rangle_{L^2}}{\langle \sigma_3\vec{\eta}_{\gamma_0},\vec{\phi}_{\gamma_0}\rangle_{L^2}}\vec{\phi}_{\gamma_0}.
$$
Note that $P_{\gamma_0} \vec{\phi}_s = \vec{\phi}_s$ and $P_{\gamma_0} \vec{\phi}_{\gamma_0} = \vec{0}$.

From equations \eqref{deltaeq} and \eqref{decomp}, we define the operator equation
\begin{equation}\label{eqF}
F(\vec{\phi}_s,u_s,v_s,\lambda):=(\partial_x -
M_{\gamma_0})\vec{\phi}_{s}-\Delta M(\vec{\phi}_{\gamma_0} + \vec{\phi}_s)=0.
\end{equation}
Clearly, since $\dim \ker(\partial_x - M_{\gamma_0}) = 1 \neq 0$,
the Fr\'{e}chet derivative $D_{\vec{\phi}_s} F(0,0,0,\lambda_0)=\partial_x-M_{\gamma_0}$
has no bounded inverse. Let $\hat{P}_{\gamma_0}=\sigma_3 P_{\gamma_0} \sigma_3$ and notice that
$$
\hat{P}_{\gamma_0}: L^2(\mathbb{R};\mathbb{C}^2)\rightarrow L^2(\mathbb{R};\mathbb{C}^2)\cap \Span\{\vec{\eta}_{\gamma_0}\}^{\perp}.
$$
We decompose equation \eqref{eqF} by the projection $\hat{P}$ into two equations
\begin{eqnarray} \label{eqG}
G(\vec{\phi}_s,u_s,v_s,\lambda)&:=& \hat{P}_{\gamma_0} F(\vec{\phi}_s,u_s,v_s,\lambda) = 0, \\
\label{bifeq}
H(\vec{\phi}_s,u_s,v_s,\lambda) &:=& (I - \hat{P}_{\gamma_0}) F(\vec{\phi}_s,u_s,v_s,\lambda) = 0.
\end{eqnarray}

\vspace{0.2cm}

\emph{Step 3.}  First, we note that since
$\dim \ker(\partial_x-M_{\gamma_0})=\dim\ker(\partial_x+M_{\gamma_0}^*) = 1 <\infty$,
then $\partial_x-M_{\gamma_0}$ is a Fredholm operator of index zero.
Observe that $\Range(G)= L^2(\mathbb{R};\mathbb{C}^2)\cap \mbox{span}\{\vec{\eta}_{\gamma_0}\}^{\perp}$,
where $\vec{\eta}_{\gamma_0} \in \ker\{\partial_x+M_{\gamma_0}^*\}$.
By the Fredholm alternative theorem, $D_{\vec{\phi}_s} G(0,0,0,\lambda_0) =
\hat{P}_{\gamma_0} D_{\vec{\phi}_s} F(0,0,0,\lambda_0)$ has a bounded inverse,
\begin{equation}
\label{bounded-inverse}
(\partial_x-M_{\gamma_0})^{-1}:L^2(\mathbb{R};\mathbb{C}^2) \cap \Span\{\vec{\eta}_{\gamma_0}\}^{\perp}\rightarrow
H^1(\mathbb{R};\mathbb{C}^2) \cap \mbox{span}\{\sigma_3 \vec{\eta}_{\gamma_0}\}^{\perp}.
\end{equation}

Next we claim that for some $(u_s,v_s) \in L^2(\mathbb{R})$ and $\lambda\in \mathbb{C}$,
there exists a unique $\vec{\phi}_* \in H^1(\mathbb{R};\mathbb{C}^2)$ such that $G(\vec{\phi}_*,u_s,v_s,\lambda)=0$.
This can be done by the implicit function theorem. The function
$$
G : H^1(\mathbb{R};\mathbb{C}^2) \times L^2(\mathbb{R};\mathbb{C})\times L^2(\mathbb{R};\mathbb{C})\times
\mathbb{C} \rightarrow L^2(\mathbb{R};\mathbb{C}^2)\cap \mbox{span}\{\vec{\eta}_{\gamma_0}\}^{\perp}
$$
is $C^1$ in $u_s$, $v_s$ (and their complex conjugates), $\lambda$, and $\vec{\phi}_s$.
We also find that $G(0,0,0,\lambda_0)=0$ and the derivative $D_{\vec{\phi}_s}G(0,0,0,\lambda_s)$ is invertible
with the bounded inverse (\ref{bounded-inverse}). For some $\epsilon,\rho >0$, let
\begin{eqnarray*}
U_{\epsilon}&:=&\{(u_s,v_s,\lambda)\in L^2(\mathbb{R}) \times L^2(\mathbb{R}) \times \mathbb{C} : \quad
\|u_s\|_{L^2}+\|v_s\|_{L^2}+ |\lambda-\lambda_0|<\epsilon\},\\
V_{\rho}&:=&\{\vec{\phi}_s\in H^1(\mathbb{R};\mathbb{C}^2)\cap\text{span}\{\sigma_3\vec{\eta}_{\gamma_0}\}^{\perp}:
\quad \|\vec{\phi}_s\|_{H^1} \leq \rho\}.
\end{eqnarray*}
Then, by the implicit function theorem, for sufficiently small $\epsilon,\rho >0$,
and for each $(u_s,v_s,\lambda)\in U_{\epsilon}$,
there exists a unique $\vec{\phi}_* \in V_{\rho}$ such that $G(\vec{\phi}_*, u_s,v_s,\lambda)=0$.

A unique element $\vec{\phi}_*$ depends implicitly on  $(u_s,v_s,\lambda)$, that is,
we can write $\vec{\phi}_*:=\vec{\phi}_*(u_s,v_s,\lambda)$. From equations (\ref{eqF}) and
\eqref{eqG}, we have
\begin{equation}\label{eqG2}
(I - P_{\gamma_0} (\partial_x - M_{\gamma_0})^{-1} \hat{P}_{\gamma_0} \Delta M) \vec{\phi}_* =
P_{\gamma_0} (\partial_x - M_{\gamma_0})^{-1} \hat{P}_{\gamma_0} \Delta M \vec{\phi}_{\gamma_0}
\end{equation}
and from boundedness of the inverse operator given by (\ref{bounded-inverse})
and inequality \eqref{Lbound}, we obtain
\begin{equation}\label{preeq}
\|\vec{\phi}_*\|_{H^1} \lesssim \| P_{\gamma_0} (\partial_x - M_{\gamma_0})^{-1} \hat{P}_{\gamma_0} \Delta M \vec{\phi}_{\gamma_0}\|_{H^1}
\lesssim \|\Delta M \vec{\phi}_{\gamma_0} \|_{L^2} \lesssim |\lambda-\lambda_0| + \|u_s\|_{L^2}+\|v_s\|_{L^2},
\end{equation}
if $(u_s,v_s,\lambda)\in U_{\epsilon}$.

\vspace{0.2cm}

\emph{Step 4.}  Lastly we address the bifurcation equation \eqref{bifeq}
to determine $\lambda \in \mathbb{C}$. From equations (\ref{eqF}) and \eqref{bifeq},
the bifurcation equation can be written explicitly as
\begin{equation}\label{bifeq2}
I(u_s,v_s,\lambda) := \langle \vec{\eta}_{\gamma_0}, \Delta M (\vec{\phi}_{\gamma_0} + \vec{\phi}_*(u_s,v_s,\lambda)) \rangle_{L^2} = 0,
\end{equation}
where $\vec{\phi}_*(u_s,v_s,\lambda)$ is uniquely expressed from (\ref{eqG2}) if $(u_s,v_s,\lambda)\in U_{\epsilon}$.

By using the explicit expression (\ref{explicit-M}), we check that
$s := \partial_{\lambda} I(0,0,\lambda_0) \neq 0$, where
\begin{eqnarray*}
s & = & \frac{i}{2} \langle \vec{\eta}_{\gamma_0}, \left[\begin{matrix}(\lambda_0 + \lambda_0^{-3}) &
-(\overline{u}_{\gamma_0} \lambda_0^{-2} + \overline{v}_{\gamma_0} ) \\ -(u_{\gamma_0} \lambda_0^{-2} +v_{\gamma_0} ) &
-(\lambda_0+\lambda_0^{-3}) \end{matrix}\right] \vec{\phi}_{\gamma_0} \rangle_{L^2} \\
&=& i e^{-i \gamma_0/2}  \int_{\mathbb{R}} \left( 2 \cos \gamma_0 \left|\sech\left(x\sin\gamma_0-i\frac{\gamma_0}{2}\right)\right|^2
+ \sin^2\gamma_0 \left|\sech\left(x\sin\gamma_0-i\frac{\gamma_0}{2}\right)\right|^4 \right) dx \\
&=& 4 i e^{-i \gamma_0/2}  \int_{\mathbb{R}} \frac{1 + \cos \gamma_0 \cosh(2 x \sin \gamma_0)}{(\cosh(2 x \sin \gamma_0) + \cos \gamma_0)^2} dx \\
& = & \frac{4 i e^{-i\gamma_0/2}}{\sin \gamma_0}.
\end{eqnarray*}
As a result, equation \eqref{bifeq2} can be used to uniquely determine the spectral parameter $\lambda$
if $(u_s,v_s,\lambda)\in U_{\epsilon}$.
From inequalities \eqref{Lbound} and \eqref{preeq}, we obtain that this $\lambda$
satisfies the bound
\begin{equation} \label{preeq2}
|\lambda-\lambda_0| \lesssim  \|u_s\|_{L^2}+\|v_s\|_{L^2}.
\end{equation}
With inequalities \eqref{preeq} and \eqref{preeq2}, the proof of Lemma \ref{onebound1} is complete.
\end{proof}

\begin{remark}
A spectral parameter $\lambda$ in Lemma \ref{onebound1} may not be on the unit circle $|\lambda|=1$
even if $\lambda_0 = e^{i \gamma_0/2}$ is on the unit circle. In what follows,
we develop the theory when $\lambda$ occurs on the unit circle, hence
we write $\lambda=e^{i\gamma/2}$ for some $\gamma \in (0,\pi)$.
All results obtained below can be generalized to the case of $|\lambda| \neq 1$
by using the Lorentz transformation in Proposition \ref{prop-Lorentz}.
\label{remark-normalization}
\end{remark}

In Lemma \ref{bound2} below, we will show that a solution $\vec{\phi}$
determined in the proof of Lemma \ref{onebound1} can be written explicitly as the perturbed
solution around $\vec{\phi}_{\gamma}$ in suitable function spaces. Then, in
Lemma \ref{pqbound} below, we will use this representation and the
auto-B\"{a}cklund transformation \eqref{up} and \eqref{vq} to show that
$(p_0,q_0)$ are small in $L^2$ norm.

To develop this analysis, we first prove several technical results.
Let $(u,v) = (u_{\gamma},v_{\gamma})$, $\lambda = e^{i \gamma/2}$
and consider the linear inhomogeneous equation
\begin{equation}\label{inhomeq}
\partial_x \vec{w}- M_{\gamma}\vec{w}=\vec{f},
\end{equation}
where
$$
M_{\gamma} = \frac{1}{2} \left[\begin{matrix} -\sin\gamma & i(e^{-i\gamma/2}\overline{u}_{\gamma}-e^{i\gamma/2}\overline{v}_{\gamma} ) \\ i(e^{-i\gamma/2}u_{\gamma}-e^{i\gamma/2} v_{\gamma} ) & \sin\gamma\end{matrix}\right].
$$
We introduce Banach spaces $X=X_1\times X_2$ and $Y=Y_1\times Y_2$
such that for $\vec{w}=(w_1,w_2)^t\in X$ and $\vec{f}=(f_1,f_2)^t\in Y$, we have
$$
\|\vec{w}\|_X := \|w_1\|_{X_1}+\|w_2\|_{X_2}, \quad \|\vec{f}\|_Y := \|f_1\|_{Y_1}+\|f_2\|_{Y_2},
$$
where
\begin{eqnarray*}
\|w_1\|_{X_1} & := & \inf_{w_1=v_1+u_1}\left( \|v_1e^{-\frac{x}{2}\sin\gamma}|\cosh(x\sin\gamma -i\gamma/2)|\|_{L^{\infty}_x} \right. \\
& \phantom{t} & \left. \phantom{texttexttexttext} +\|u_1e^{\frac{x}{2}\sin\gamma}|\cosh(x\sin\gamma -i\gamma/2)|\|_{L^2_x \cap L^{\infty}_x}  \right), \\
\|w_2\|_{X_2} & := & \inf_{w_2=v_2+u_2}\left(\|v_2e^{\frac{x}{2}\sin\gamma}|\cosh(x\sin\gamma -i\gamma/2)|\|_{L^{\infty}_x} \right. \\
& \phantom{t} & \left. \phantom{texttexttexttext} + \|u_2e^{-\frac{x}{2}\sin\gamma}|\cosh(x\sin\gamma -i\gamma/2)|\|_{L^2_x \cap L^{\infty}_x}  \right)
\end{eqnarray*}
and
\begin{eqnarray*}
\|f_1\|_{Y_1} & := & \inf_{f_1=g_1+h_1}\left( \|g_1e^{\frac{x}{2}\sin\gamma}|\cosh(x\sin\gamma -i\gamma/2)|\|_{L^2_x} \right. \\
& \phantom{t} & \left. \phantom{texttexttexttext} +\|h_1e^{-\frac{x}{2}\sin\gamma}|\cosh(x\sin\gamma -i\gamma/2)|\|_{L^{2}_x \cap L^1_x} \right), \\
\|f_2\|_{Y_2} & := & \inf_{f_2=g_2+h_2} \left( \|g_2e^{-\frac{x}{2}\sin\gamma}|\cosh(x\sin\gamma -i\gamma/2)|\|_{L^2_x} \right. \\
& \phantom{t} & \left. \phantom{texttexttexttext} +\|h_2e^{\frac{x}{2}\sin\gamma}|\cosh(x\sin\gamma -i\gamma/2)|\|_{L^2_x \cap L^1_x} \right).
\end{eqnarray*}
It is obvious that $X$ and $Y$ are continuously embedded into $L^2(\mathbb{R})$.
We shall estimate the bound of the operator $P_{\gamma} (\partial_x-M_{\gamma})^{-1} \hat{P}_{\gamma} : Y \to X$,
where projection operators $P_{\gamma}$ and $\hat{P}_{\gamma}$ are defined in the proof
of Lemma \ref{onebound1}.
First, we will obtain an explicit solution $\vec{w}\in  H^1(\mathbb{R};\mathbb{C}^2)\cap \Span\{\sigma_3\vec{\eta}_{\gamma}\}^{\perp}$
for the linear inhomogeneous equation \eqref{inhomeq}
when $\vec{f} \in L^2(\mathbb{R};\mathbb{C}^2)\cap \ker(\partial_x + M_{\gamma}^*)^{\perp}$.
Then, we will prove that the mapping $Y \ni \vec{f} \mapsto \vec{w} \in X$ is bounded.
These goals are achieved in the next two lemmas.

\begin{lem} \label{ode}
For any $\vec{f}=(f_1,f_2)^t\in L^2(\mathbb{R};\mathbb{C}^2)\cap \Span\{ \vec{\eta}_{\gamma}\}^{\perp}$,
there exists a unique solution $\vec{w}\in  H^1(\mathbb{R};\mathbb{C}^2)\cap \Span\{\sigma_3\vec{\eta}_{\gamma}\}^{\perp}$
of the inhomogeneous equation (\ref{inhomeq}) that can be written as
\begin{equation}\label{wsol}
\vec{w}(x) = \frac{1}{4} \vec{\phi}_{\gamma}(x) \left[ k(\vec{f}) + W_-(x) + W_+(x) \right]  +
\frac{1}{4}\vec{\xi}_{\gamma}(x) \int_{-\infty}^{x} \vec{\eta}_{\gamma}(y) \cdot \vec{f}(y)  dy,
\end{equation}
where $k(\vec{f})$ is a continuous linear functional on $L^2(\mathbb{R};\mathbb{C}^2)$ and we have denoted
\begin{eqnarray*}
W_-(x) & := & \int_{-\infty}^{x}e^{-\frac{1}{2}y\sin\gamma}(e^{2y\sin\gamma}+2\cos\gamma+y\sin(2\gamma))
\left|\sech\left(y \sin\gamma -i \frac{\gamma}{2} \right)\right| f_1(y)dy, \\
W_+(x) & := & \int_{x}^{\infty}e^{-\frac{3}{2}y\sin\gamma}(-1+e^{2y\sin\gamma}y\sin(2\gamma))
\left|\sech\left(y \sin\gamma -i \frac{\gamma}{2} \right)\right| f_2(y)dy.
\end{eqnarray*}
\end{lem}

\begin{proof}
Since $\partial_x - M_{\gamma}:H^1(\mathbb{R};\mathbb{C}^2)\rightarrow L^2(\mathbb{R};\mathbb{C}^2)$
is a Fredholm operator of index zero and $\ker(\partial_x + M_{\gamma}^*) = \Span\{ \vec{\eta}_{\gamma}\}$,
the inhomogeneous equation (\ref{inhomeq}) has a solution in $H^1(\mathbb{R};\mathbb{C}^2)$ if
and only if $\vec{f}\in L^2(\mathbb{R};\mathbb{C}^2) \cap \Span\{ \vec{\eta}_{\gamma}\}^{\perp}$.
For uniqueness, we add the constraint $\vec{w} \in \Span\{ \sigma_3 \vec{\eta}_{\gamma}\}^{\perp}$.

Recall that $U=[\vec{\phi}_{\gamma}, \vec{\xi}_{\gamma}]$ is a
fundamental matrix of the homogeneous equation $(\partial_x - M_{\gamma}) U = 0$.
All functions are known explicitly as
$$
\vec{\phi}_{\gamma}(x) = \left[ \begin{matrix} e^{\frac{x}{2}\sin\gamma} \\ e^{-\frac{x}{2}\sin\gamma} \end{matrix}\right] Q(x), \quad
\vec{\eta}_{\gamma}(x) = \left[ \begin{matrix} e^{-\frac{x}{2}\sin\gamma} \\ -e^{\frac{x}{2}\sin\gamma} \end{matrix}\right] Q(x),
$$
and
$$
\vec{\xi}_{\gamma}(x) = \left[ \begin{matrix} e^{\frac{x \sin\gamma}{2}}(e^{-2x \sin\gamma}-x\sin(2\gamma))
\\ -e^{-\frac{x\sin\gamma}{2}}(e^{2x\sin\gamma}+2\cos\gamma+x\sin(2\gamma)) \end{matrix}\right] Q(x),
$$
where
\begin{eqnarray*}
Q(x) := \left|\sech\left(x\sin\gamma -i \frac{\gamma}{2} \right)\right|.
\end{eqnarray*}
From variation of parameters, we have the explicit representation (\ref{wsol}),
where $k(\vec{f})$ is the constant of integration and the other constant is set to zero to ensure that
$\vec{w}\in H^1(\mathbb{R};\mathbb{C}^2)$. It remains to prove that every term in
the explicit expression \eqref{wsol} belongs to $L^2(\mathbb{R};\mathbb{C}^2)$.

Since $|\vec{\phi}_{\gamma}(x)| \lesssim e^{-\frac{|x|}{2}\sin\gamma}$ and $|Q(x)| \lesssim e^{-|x|\sin\gamma}$ for all $x\in \mathbb{R}$,
we have
\begin{eqnarray*}
\| W_- \vec{\phi}_{\gamma} \|_{L^2} &\lesssim& \left\|e^{-\frac{|x|}{2}\sin\gamma}\int_{-\infty}^{x} e^{-\frac{1}{2}y\sin\gamma}(e^{2y\sin\gamma}+2\cos\gamma+y\sin(2\gamma))Q(y)f_1(y)dy\right\|_{L^2_x}\\
&\lesssim& \left\|\int_{-\infty}^{x}e^{\frac{1}{2}(y-x)\sin\gamma} |f_1(y)| dy\right\|_{L^2_x} +
\left\|e^{-\frac{|x|}{2}\sin\gamma}\int_{-\infty}^{x}e^{-\frac{1}{2}|y|\sin\gamma} (2+|y|) |f_1(y)|dy\right\|_{L^2_x}\\
&\lesssim& \|\vec{f}\|_{L^2}.
\end{eqnarray*}
and
\begin{eqnarray*}
\| W_+ \vec{\phi}_{\gamma} \|_{L^2} &\lesssim& \left\|e^{-\frac{|x|}{2}\sin\gamma} \int_{x}^{\infty}e^{-\frac{3}{2}y\sin\gamma}(-1+e^{2y\sin\gamma}y\sin(2\gamma))Q(y)f_2(y)dy \right\|_{L^2_x}\\
&\lesssim& \left\|\int_{x}^{\infty}e^{\frac{1}{2}(x-y)\sin\gamma} |f_2(y)| dy\right\|_{L^2_x}+
\left\|e^{-\frac{|x|}{2}\sin\gamma}\int_x^{\infty}e^{-\frac{1}{2}|y|\sin\gamma} |y| |f_2(y)|dy\right\|_{L^2_x}\\
&\lesssim& \|\vec{f}\|_{L^2},
\end{eqnarray*}
where notation $\| f(x) \|_{L^2_x}$ is used in place of $\| f(\cdot) \|_{L^2}$.
Since $\vec{f}\in L^2(\mathbb{R};\mathbb{C}^2) \cap \Span\{ \vec{\eta}_{\gamma}\}^{\perp}$, then
$$
\int_{x}^{\infty} \vec{\eta}_{\gamma}(y) \cdot \vec{f}(y) dy =
-\int_{-\infty}^{x} \vec{\eta}_{\gamma}(y) \cdot \vec{f}(y) dy.
$$
Using this equality, we can estimate the last term in the explicit expression \eqref{wsol} as follows
\begin{eqnarray*}
\left\|\vec{\xi}_{\gamma}(x) \int_{- \infty}^{x}\vec{\eta}_{\gamma}(y) \cdot \vec{f}(y) dy\right\|_{L^2_x} &\lesssim &
\left\| e^{-\frac{x}{2}\sin\gamma}\int_{- \infty}^{x}\vec{\eta}_{\gamma}(y) \cdot \vec{f}(y) dy\right\|_{L^2_x} +
\left\| e^{\frac{x}{2}\sin\gamma}\int_x^{\infty}\vec{\eta}_{\gamma}(y) \cdot \vec{f}(y) dy \right\|_{L^2_x} \\
&\lesssim& \left\| \int_{- \infty}^{x} e^{\frac{1}{2}(y-x)\sin\gamma}|\vec{f}(y)|dy\right\|_{L^2_x} +
\left\| \int_x^{\infty}e^{-\frac{1}{2}(y-x)\sin\gamma}|\vec{f}(y)|dy \right\|_{L^2_x}\\
&\lesssim&  \|\vec{f}\|_{L^2},
\end{eqnarray*}
where $|\vec{f}|$ is the vector norm of the $2$-vector $\vec{f}$.
Since $\langle \sigma_3 \vec{\eta}_{\gamma}, \vec{\phi}_{\gamma} \rangle_{L^2} \neq 0$,
$k(\vec{f})$ is uniquely determined from the orthogonality condition
$\langle \sigma_3 \vec{\eta}_{\gamma}, \vec{w} \rangle_{L^2} = 0$. Since
all other terms in (\ref{wsol}) are in $L^2(\mathbb{R};\mathbb{C}^2)$,
$k(\vec{f})$ is bounded for all $\vec{f} \in L^2(\mathbb{R};\mathbb{C}^2)$.
Therefore, $k(\vec{f})$ is a continuous linear functional on $L^2(\mathbb{R};\mathbb{C}^2)$.
\end{proof}

\begin{lem}\label{wfbound}
Let $\vec{f} \in Y \cap \Span\{ \vec{\eta}_{\gamma}\}^{\perp}$ and
let $\vec{w}$ be a solution of the inhomogeneous equation (\ref{inhomeq}) in
Lemma \ref{ode}. Then there is a $\vec{f}$-independent constant $C>0$
such that $\|\vec{w}\|_X\leq C\|\vec{f}\|_Y$.
\end{lem}

\begin{proof}
The solution $\vec{w}$ is given by the explicit formula \eqref{wsol}.
We assume now that $\vec{f}$ belongs to the exponentially weighted space $Y$ and prove that
$\vec{w}$ belongs to the exponentially weighted space $X$.

Since $\|a\vec{\phi}_{\gamma}\|_X\leq 2\|a\|_{L^{\infty}}$ for any $a\in L^{\infty}(\mathbb{R})$,
$k(\vec{f})$ is a continuous linear functional on $L^2(\mathbb{R};\mathbb{C}^2)$,
and $Y$ is embedded into $L^2(\mathbb{R};\mathbb{C}^2)$, we have
$$
\|k(\vec{f})\vec{\phi}_{\gamma}\|_X\lesssim |k(\vec{f})|\lesssim \|\vec{f}\|_{L^2}\lesssim\|\vec{f}\|_Y.
$$
The second term in \eqref{wsol} is estimated by
\begin{eqnarray*}
\|W_- \vec{\phi}_{\gamma} \|_X &\lesssim&  \left\|\int_{-\infty}^{x}e^{-\frac{1}{2}y\sin\gamma}(e^{2y\sin\gamma}+2\cos\gamma+y\sin(2\gamma))Q(y)f_1(y)dy\right\|_{L^{\infty}_x}\\
 &\lesssim& \inf_{f_1=g_1+h_1}\left(\|g_1e^{-\frac{1}{2}x \sin\gamma}|\cosh(x\sin \gamma-i\gamma/2)|\|_{L^1_x}
 + \|h_1e^{\frac{1}{2}x \sin\gamma}|\cosh(x\sin \gamma-i\gamma/2)|\|_{L^2_x}\right) \\
 &\leq& \|\vec{f}\|_{Y_1}.
 \end{eqnarray*}
Similarly, the third term in \eqref{wsol} is estimated by $\|W_+ \vec{\phi}_{\gamma} \|_X\lesssim \|\vec{f}\|_{Y_2}$.
The last term in \eqref{wsol} is estimated as follows:
$$
\left\|\vec{\xi}_{\gamma}\int_{-\infty}^{x}\vec{\eta}_{\gamma}(y) \cdot \vec{f}(y) dy\right \|_X \leq N_1+N_2+N_3+N_4,
$$
where
\begin{eqnarray*}
N_1&=&\left\|e^{-x\sin\gamma}\int_{-\infty}^{x} \vec{\eta}_{\gamma}(y) \cdot \vec{f}(y) dy\right\|_{L^{\infty}_x \cap L^2_x},\\
N_2&=&\left\|x\sin(2\gamma)\int_{-\infty}^{x} \vec{\eta}_{\gamma}(y) \cdot \vec{f}(y) dy\right\|_{L^{\infty}_x},\\
N_3&=&\left\|e^{x\sin\gamma}\int_x^{\infty} \vec{\eta}_{\gamma}(y) \cdot \vec{f}(y) dy\right\|_{L^{\infty}_x \cap L^2_x},\\
N_4&=&\left\|(2\cos\gamma+x\sin(2\gamma))\int_{- \infty}^{x} \vec{\eta}_{\gamma}(y) \cdot \vec{f}(y) dy\right\|_{L^{\infty}_x}.
\end{eqnarray*}
Since $|\vec{\eta}_{\gamma}(x)|\lesssim e^{-\frac{|x|}{2}\sin\gamma}$ and
$\|e^{\frac{|x|}{2}\sin\gamma}\vec{f}\|_{L^2_x} \lesssim \|\vec{f}\|_Y$ for all $x\in \mathbb{R}$, we have
\begin{eqnarray*}
N_1&\lesssim& \left\|\int_{-\infty}^x e^{-(y-x)\sin\gamma} e^{\frac{|y|}{2}\sin\gamma}(|f_1|+|f_2|)dy \right\|_{L^{\infty}_x \cap L^2_x} \\
&\leq& \|e^{\frac{|x|}{2}\sin\gamma}f_1\|_{L^2_x}+\|e^{\frac{|x|}{2}\sin\gamma}f_2\|_{L^2_x} \\
&\lesssim& \|\vec{f}\|_Y.
\end{eqnarray*}
The other terms $N_2$, $N_3$, and $N_4$ are estimated similarly. All together, these estimates justify
the bound $\|\vec{w}\|_X\leq C\|\vec{f}\|_Y$ for a $\vec{f}$-independent positive constant $C$.
\end{proof}

\begin{lem} \label{bound2}
Under the condition (\ref{bound-eigenvalue-before}), assume that $\lambda = e^{i \gamma/2}$ is the eigenvalue
of the spectral problem (\ref{laxeq2}) for the eigenvector $\vec{\psi}\in H^1(\mathbb{R};\mathbb{C}^2)$
determined in Lemma \ref{onebound1}. Then, the eigenvector can be written in the form (\ref{unitary}) with
\begin{equation}
\label{representation-psi}
\vec{\phi}(x) =
\left[\begin{matrix}e^{\frac{x\sin\gamma}{2}}(1+r_{11}(x))+e^{-\frac{x\sin\gamma}{2}}r_{12}(x)\\
e^{\frac{x\sin\gamma}{2}}r_{21}(x)+e^{-\frac{x\sin\gamma}{2}}(1+r_{22}(x)) \end{matrix} \right]
\left|\sech\left(x\sin\gamma -i \frac{\gamma}{2} \right)\right|,
\end{equation}
where components $r_{ij}$ for $1 \leq i,j \leq 2$ satisfy the bound
\begin{equation}
\label{bound-on-r}
\| r_{11}\|_{L^{\infty}}+\|r_{12}\|_{L^2\cap L^{\infty}} +\|r_{21}\|_{L^2\cap L^{\infty}}+\|r_{22}\|_{L^{\infty}} \lesssim \|u_0-u_{\gamma}\|_{L^2}+\|v_0-v_{\gamma}\|_{L^2}.
\end{equation}
\end{lem}

\begin{proof}
Recall the projection operators $P_{\gamma} : L^2(\mathbb{R};\mathbb{C}^2) \rightarrow
L^2(\mathbb{R};\mathbb{C}^2)\cap \Span\{ \sigma_3 \vec{\eta}_{\gamma}\}^{\perp}$
and $\hat{P}_{\gamma} : L^2(\mathbb{R};\mathbb{C}^2) \rightarrow
L^2(\mathbb{R};\mathbb{C}^2)\cap \Span\{ \vec{\eta}_{\gamma}\}^{\perp}$
introduced in the proof of Lemma \ref{onebound1}.
The existence of the eigenvector $\vec{\phi} \in H^1(\mathbb{R};\mathbb{C}^2)$ of the spectral problem (\ref{sys})
for the eigenvalue $\lambda = e^{i \gamma/2}$ has been established in Lemma \ref{onebound1}.
Therefore, we are using operators $P_{\gamma}$ and $\hat{P}_{\gamma}$ to prove additional properties of
the eigenvector $\vec{\phi}$.

Using the projection operator $P_{\gamma}$, we decompose $\vec{\phi}=\vec{\phi}_{\gamma}+\vec{\phi}_s$
and rewrite the spectral problem (\ref{sys}) in the form
\begin{equation} \label{eqDL}
\left( \partial_x - M_{\gamma} \right) \vec{\phi}_s = \Delta \widetilde{M}(\vec{\phi}_{\gamma}+\vec{\phi}_s),
\end{equation}
where $\Delta \widetilde{M}$ is the anti-diagonal matrix that contains the perturbation terms $u_0 - u_{\gamma}$
and $v_0 - v_{\gamma}$ only. Because  $\vec{\phi}_s \in H^1(\mathbb{R};\mathbb{C}^2)$ exists by Lemma \ref{onebound1},
we realize that $\Delta \widetilde{M} (\vec{\phi}_{\gamma}+\vec{\phi}_s) =
\hat{P} \Delta \widetilde{M} (\vec{\phi}_{\gamma}+\vec{\phi}_s)$,
which yields equivalently the constraint
\begin{equation}
\label{LS-constraint}
\langle \vec{\eta}_{\gamma}, \Delta \widetilde{M} (\vec{\phi}_{\gamma}+\vec{\phi}_s) \rangle_{L^2} = 0.
\end{equation}
Therefore, we write the perturbed equation \eqref{eqDL} in the form
\begin{equation}\label{MPeq}
\vec{\phi}_{s} = P_{\gamma} \left(\partial_x - M_{\gamma} \right)^{-1} \hat{P}_{\gamma}
\Delta \widetilde{M}(\vec{\phi}_{\gamma} + \vec{\phi}_{s}).
\end{equation}
Note that the operator $\hat{P}_{\gamma}$ applies to the sum of the two terms in the right-hand-side of (\ref{MPeq})
thanks to (\ref{LS-constraint}) and cannot be applied to each term separately.

Since $\Delta \widetilde{M}$ is anti-diagonal, for any $\vec{\zeta} = (\zeta_1,\zeta_2)^t \in X$,
we have
$$
\|\Delta \widetilde{M}\vec{\zeta}\|_Y =
\|(\Delta \widetilde{M})_{1,2}\zeta_2\|_{Y_1} + \|(\Delta \widetilde{M})_{2,1}\zeta_1\|_{Y_2},
$$
which is bounded as follows:
\begin{eqnarray}
\label{bound-1-tech}
\|(\Delta \widetilde{M})_{1,2}\zeta_2\|_{Y_1} &\lesssim&  (\|u_0 - u_{\gamma}\|_{L^2}+\|v_0 - v_{\gamma}\|_{L^2})\|\zeta_2\|_{X_2},\\
\label{bound-2-tech}
\|(\Delta \widetilde{M})_{2,1}\zeta_1\|_{Y_2} &\lesssim& (\|u_0 - u_{\gamma}\|_{L^2}+\|v_0 - v_{\gamma}\|_{L^2})\|\zeta_1\|_{X_1}.
\end{eqnarray}
Bound (\ref{bound-1-tech}) follows simply from
\begin{eqnarray*}
\|(\Delta \widetilde{M})_{1,2}\zeta_2\|_{Y_1} &\leq& \inf_{\zeta_2=\xi_2 + \eta_2}
\left(\|(\Delta \widetilde{M})_{1,2} \xi_2 e^{\frac{x}{2} \sin\gamma} R(x) \|_{L^2_x}
+ \|(\Delta \widetilde{M})_{1,2} \eta_2 e^{-\frac{x}{2}\sin\gamma} R(x) \|_{L^2_x \cap L^1_x} \right)\\
&\lesssim&
\|(\Delta \widetilde{M})_{1,2}\|_{L^2} \inf_{\zeta_2 = \xi_2 +\eta_2}
\left(\|\xi_2 e^{\frac{x}{2}\sin\gamma}R(x)\|_{L^{\infty}_x}
+\|\eta_2 e^{-\frac{x}{2}\sin\gamma} R(x)\|_{L^{\infty}_x \cap L^2_x} \right)\\
& = &\|(\Delta \widetilde{M})_{1,2}\|_{L^2}\|\zeta_2\|_{X_2},
\end{eqnarray*}
where $R(x) = \left|\cosh\left(x\sin\gamma-i \frac{\gamma}{2}\right)\right|$. Bound (\ref{bound-2-tech})
is obtained similarly. Because $\vec{\phi}_{\gamma} \in X$, the bound
$\|\vec{w}\|_X\lesssim \|\vec{f}\|_Y$ in Lemma \ref{wfbound}
and bounds \eqref{bound-1-tech} and \eqref{bound-2-tech} imply
\begin{eqnarray*}
\| P_{\gamma} \left(\partial_x - M_{\gamma} \right)^{-1} \hat{P}_{\gamma}
\Delta \widetilde{M}(\vec{\phi}_{\gamma} + \vec{\phi}_{s}) \|_X
& \lesssim & \| \Delta \widetilde{M}(\vec{\phi}_{\gamma} + \vec{\phi}_{s}) \|_{Y} \\
& \lesssim & \left(\|u_0 - u_{\gamma}\|_{L^2}+\|v_0 - v_{\gamma}\|_{L^2} \right) (1 + \| \vec{\phi}_s \|_X ).
\end{eqnarray*}
Since $\|u_0 - u_{\gamma}\|_{L^2} + \|v_0 - v_{\gamma}\|_{L^2}$ is sufficiently small,
the component $\vec{\phi}_s$ in (\ref{MPeq}) satisfies the bound
\begin{equation}
\label{bound-3-tech}
\|\vec{\phi}_{s}\|_X \lesssim \|u_0 - u_{\gamma}\|_{L^2} + \|v_0 - v_{\gamma}\|_{L^2}.
\end{equation}
This completes the proof of the bound (\ref{bound-on-r}) in the representation (\ref{representation-psi}),
because the bound (\ref{bound-3-tech}) on $\vec{\phi}_s$ in Banach space $X$ yields the bounds on the components
$r_{ij}$ in the corresponding spaces.
\end{proof}

\begin{cor} \label{bound3}
In addition to the assumptions of Lemma \ref{bound2},
assume that $(u_0,v_0) \in H^m(\mathbb{R})$ for an integer $m \geq 0$.
Then, $r_{ij}$ for $1 \leq i,j \leq 2$
defined by \eqref{representation-psi} are $C^m$-functions of $x$.
\end{cor}

\begin{proof}
The statement is proved for $m = 0$ in Lemma \ref{bound2}, because $r_{ij}$ are bounded functions
according to the bound (\ref{bound-on-r}) and they are continuous functions
since $\vec{\phi} \in H^1(\mathbb{R};\mathbb{C}^2)$.

For $m = 1$, we differentiate the equation \eqref{eqDL} with respect to $x$ to get
\begin{equation} \label{eqDL-1-deriv}
\left( \partial_x - M_{\gamma} \right) \partial_x\vec{\phi}_s = \vec{r} +
R\vec{\phi}_s+\Delta \widetilde{M}\partial_x\vec{\phi}_s,
\end{equation}
where $\vec{r} := \partial_x(\Delta \widetilde{M} \vec{\phi}_{\gamma})$ and
$R := \partial_x(M_{\gamma})+\partial_x( \Delta \widetilde{M})$.
Recall that $\vec{\phi}_s \in X$ by Lemma \ref{bound2}.
If $(u_0,v_0) \in H^1(\mathbb{R})$, then $\vec{r} + R \vec{\phi}_s \in Y$
according to the bounds
\begin{eqnarray*}
\|\vec{r}\|_Y &\lesssim& \|u_0-u_{\gamma}\|_{H^1}+\|v_0-v_{\gamma}\|_{H^1}, \\
\|R \vec{\phi}_s\|_Y &\lesssim&(\|u_0-u_{\gamma}\|_{H^1}+\|v_0-v_{\gamma}\|_{H^1}) \|\vec{\phi}_s\|_X.
\end{eqnarray*}
Since $\vec{r} + R \vec{\phi}_s \in Y$ and
$\partial_x \vec{\phi}_s \in L^2(\mathbb{R})$ from solution of (\ref{MPeq}),
we have
$$
\vec{r} + R\vec{\phi}_s+\Delta \widetilde{M}\partial_x\vec{\phi}_s =
\hat{P}_{\gamma}(\vec{r} + R\vec{\phi}_s+\Delta \widetilde{M}\partial_x\vec{\phi}_s)
$$
Therefore, we can write the derivative equation \eqref{eqDL-1-deriv} in the form
\begin{equation}\label{MPeq-1-deriv}
\partial_x\vec{\phi}_{s} = P_{\gamma} \left(\partial_x - M_{\gamma} \right)^{-1}\hat{P}_{\gamma}
(\vec{r} + R\vec{\phi}_s+\Delta \widetilde{M}\partial_x\vec{\phi}_s).
\end{equation}
Using bounds \eqref{bound-1-tech} and \eqref{bound-2-tech}, we obtain
\begin{equation}
\label{bound-4-tech}
\| \partial_x \vec{\phi}_{s}\|_X \lesssim \| \vec{r} + R \vec{\phi}_s \|_{Y} < \infty,
\end{equation}
from which it follows that $\partial_x \vec{\phi}_s \in H^1(\mathbb{R}) \cap X$, hence
$\partial_x r_{ij} \in C(\mathbb{R})$ for $1 \leq i,j \leq 2$.

For $m \geq 2$, differentiating \eqref{eqDL} $m$ times, we have the expression
\begin{equation} \label{eqDL-m-deriv}
\left( \partial_x - M_{\gamma} \right) \partial_x^m\vec{\phi}_s = \vec{r}_m
+ \Delta\widetilde{M}\partial_x^m\vec{\phi}_s,
\end{equation}
where $\vec{r}_m := \partial_x^m(\Delta\widetilde{M}\vec{\phi}_{\gamma})+[\partial_x^m,M_{\gamma}+\Delta\widetilde{M}]\vec{\phi}_s$
and we denote $[\partial_x,f]g=\partial_x(fg)-f\partial_x(g)$. We note that the term
$[\partial_x^m,M_{\gamma}+\Delta\widetilde{M}]\vec{\phi}_s$ does not contain the $m$-th
derivative of $\vec{\phi}_s$. By an induction similar to the case $m = 1$, we find
that $\vec{r}_m \in Y$ according to the bound
$$
\| \vec{r}_m \|_Y \lesssim \|u_0-u_{\gamma}\|_{H^m}+\|v_0-v_{\gamma}\|_{H^m}.
$$
Hence if $(u_0,v_0)\in H^m(\mathbb{R})$, then $\partial_x^m\vec{\phi}_s \in H^1(\mathbb{R}) \times X$,
hence $\partial_x^m r_{ij}\in C(\mathbb{R})$ for $1\leq i,j\leq 2$.
\end{proof}

\begin{lem} \label{pqbound}
Under the condition (\ref{bound-eigenvalue-before}), assume that $\lambda = e^{i \gamma/2}$ is the eigenvalue
of the spectral problem (\ref{laxeq2}) for the eigenvector $\vec{\psi}\in H^1(\mathbb{R};\mathbb{C}^2)$
determined in Lemma \ref{onebound1} and define
\begin{equation}\label{p}
p_0 := -u_0 \frac{e^{-i\gamma/2}|\psi_1|^2+e^{i\gamma/2}|\psi_2|^2}{e^{i\gamma/2}|\psi_1|^2+e^{-i\gamma/2}|\psi_2|^2} +
\frac{2i\sin\gamma \overline{\psi}_1\psi_2}{e^{i\gamma/2}|\psi_1|^2+e^{-i\gamma/2}|\psi_2|^2},
\end{equation}
\begin{equation}\label{q}
q_0 := -v_0 \frac{e^{i\gamma/2}|\psi_1|^2+e^{-i\gamma/2}|\psi_2|^2}{e^{-i\gamma/2}|\psi_1|^2+e^{i\gamma/2}|\psi_2|^2} -
\frac{2i\sin\gamma \overline{\psi}_1\psi_2}{e^{-i\gamma/2}|\psi_1|^2+e^{i\gamma/2}|\psi_2|^2}.
\end{equation}
Then, $(p_0,q_0) \in L^2(\mathbb{R})$ satisfy the bound
\begin{equation}
\label{bound-p-q}
\|p_0\|_{L^2} + \|q_0\|_{L^2} \lesssim \|u_0-u_{\gamma}\|_{L^2}+\|v_0-v_{\gamma}\|_{L^2}.
\end{equation}
If, in addition, $(u_0,v_0)\in H^m(\mathbb{R})$ for an integer $m \geq 1$,
then $(p_0,q_0) \in H^m(\mathbb{R})$.
\end{lem}

\begin{proof}
Let us rewrite equation \eqref{p} as
\begin{equation}\label{p2}
p_0 S = -u_0 + \frac{2i\sin\gamma \overline{\psi}_1\psi_2}{e^{-i\gamma/2}|\psi_1|^2+e^{i\gamma/2}|\psi_2|^2},
\end{equation}
where $S$ is a module-one factor given by
$$
S := \frac{e^{i\gamma/2}|\psi_1|^2+e^{-i\gamma/2}|\psi_2|^2}{e^{-i\gamma/2}|\psi_1|^2+e^{i\gamma/2}|\psi_2|^2}.
$$
We use the representation (\ref{unitary}) and (\ref{representation-psi})
for the eigenvector $\vec{\psi}$. Substituting $\vec{\psi}$ into the second term of \eqref{p2}, we obtain
\begin{eqnarray*}
& \phantom{t} &
\frac{2i\sin\gamma \overline{\psi}_1\psi_2}{e^{-i\gamma/2}|\psi_1|^2+e^{i\gamma/2}|\psi_2|^2} =
\frac{2i\overline{f}^2 \sin\gamma \left[ 1+\epsilon_1 + \epsilon_2 e^{x \sin \gamma}
+ \epsilon_3 e^{-x \sin \gamma} \right]}{e^{x\sin\gamma-i\frac{\gamma}{2}}(1+\epsilon_4)+
e^{-x\sin\gamma+i\frac{\gamma}{2}}(1+\epsilon_5)+\epsilon_6} \\
&=& i\overline{f}^2 \sin\gamma \sech(x\sin\gamma-i\gamma/2)
\left[ 1 + \mathcal{O}(|\epsilon_1|+|\epsilon_4|+|\epsilon_5|+|\epsilon_6|) \right] + \mathcal{O}(|\epsilon_2|+|\epsilon_3|),
\end{eqnarray*}
where $f(x) = e^{\frac{i}{4}\int_0^{x}(|u_0|^2-|v_0|^2)dx}$ and we have defined
\begin{eqnarray*}
\epsilon_1 &:=&\overline{r}_{11}+ r_{22} +\overline{r}_{11}r_{22} +\overline{r}_{12}r_{21},\\
\epsilon_2 &:=&r_{21}+\overline{r}_{11}r_{21},\\
\epsilon_3 &:=&\overline{r}_{12}+\overline{r}_{12}r_{22},\\
\epsilon_4 &:=& r_{11} + \overline{r}_{11} + |r_{11}|^2 + e^{i\gamma} |r_{21}|^2,\\
\epsilon_5 &:=& r_{22} + \overline{r}_{22} + |r_{22}|^2 + e^{-i\gamma} |r_{12}|^2,\\
\epsilon_6 &:=&2e^{-i\gamma/2}\mbox{Re}(r_{12}+ \overline{r}_{11} r_{12}) + 2e^{i\gamma/2}\mbox{Re}(r_{21}+r_{21}\overline{r}_{22}).
\end{eqnarray*}
Bound (\ref{bound-on-r}) in Lemma \ref{bound2} implies that
$$
\|\epsilon_1\|_{L^{\infty}}+\|\epsilon_2\|_{L^{\infty}\cap L^2}+\|\epsilon_3\|_{L^{\infty}\cap L^2}+\|\epsilon_4\|_{L^{\infty}}+\|\epsilon_5\|_{L^{\infty}}+\|\epsilon_6\|_{L^{\infty}\cap L^2}
\lesssim \|u_0-u_{\gamma}\|_{L^2}+\|v_0-v_{\gamma}\|_{L^2}.
$$
Since $u_{\gamma}(x) = i\sin\gamma \sech(x\sin\gamma-i\gamma/2)$ and $|f(x)|=1$ for all $x \in \mathbb{R}$,
we obtain
$$
\left\| \frac{2i \sin\gamma  \overline{\psi}_1\psi_2}{e^{-i\gamma/2}|\psi_1|^2+e^{i\gamma/2}|\psi_2|^2}
-\overline{f}^2u_{\gamma}  \right\|_{L^2} \lesssim \|u_0-u_{\gamma}\|_{L^2}+\|v_0-v_{\gamma}\|_{L^2}.
$$
Applying the triangle inequality to the representation \eqref{p2}, we obtain
\begin{eqnarray*}
\|p_0\|_{L^2} = \|p_0 S\|_{L^2} & \leq & \|u_0 -\overline{f}^2u_{\gamma}\|_{L^2}+\left\| \frac{2i \sin\gamma \overline{\psi}_1\psi_2}{e^{-i\gamma/2}|\psi_1|^2+e^{i\gamma/2}|\psi_2|^2}-\overline{f}^2u_{\gamma}  \right\|_{L^2} \\
& \lesssim & \|u_0-\overline{f}^2u_{\gamma}\|_{L^2}+\|u_0-u_{\gamma}\|_{L^2}+\|v_0-v_{\gamma}\|_{L^2}.
\end{eqnarray*}
Using the Taylor series expansion (\ref{Taylor-series-expansion}) and the triangle inequality, we obtain
\begin{eqnarray*}
\|u_0-\overline{f}^2u_{\gamma}\|_{L^2} &\leq &\|u_0-u_{\gamma}\|_{L^2}+\|u_{\gamma}\|_{L^2}\|1-\overline{f}^2\|_{L^{\infty}} \\
& \lesssim & \|u_0-u_{\gamma}\|_{L^2}+\|v_0-v_{\gamma}\|_{L^2},
\end{eqnarray*}
which finally yields the bound (\ref{bound-p-q}) for $\| p_0\|_{L^2}$. The bound
(\ref{bound-p-q}) for $\| q_0 \|_{L^2}$ is obtained in exactly the same way.

Now if $(u_0,v_0) \in H^m(\mathbb{R})$ for an integer $m \geq 1$,
we can differentiate equation (\ref{p2}) in $x$ $m$ times
and use Corollary \ref{bound3} to conclude that $(p_0,q_0) \in H^m(\mathbb{R})$.
\end{proof}

\section{From a small solution to a perturbed one-soliton solution}

Here we use the auto-B\"{a}cklund transformation given by Proposition \ref{backlund} to transform
a smooth solution of the MTM system (\ref{MTM}) in a $L^2$-neighborhood of the zero solution
to the one in a $L^2$-neighborhood of the one-soliton solution.

Let $(p_0,q_0) \in H^2(\mathbb{R})$ be the initial
data for the MTM system (\ref{MTM}), which is sufficiently small in $L^2$ norm. Let $\vec{\phi}$
be a solution of the linear equation
\begin{equation}
\label{lax-eq-phi}
\partial_x \vec{\phi} = L(p_0,q_0,\lambda) \vec{\phi}
\end{equation}
with $\lambda = e^{i\gamma/2}$. Two linearly independent solutions of
the linear equation (\ref{lax-eq-phi}) are constructed in Lemma \ref{firstbound} below.

Now, let $(p,q) \in C(\mathbb{R};H^2(\mathbb{R}))$ be the unique solution of the MTM
system (\ref{MTM}) such that $(p,q) |_{t = 0} = (p_0,q_0)$, which exists
according to Theorem \ref{theorem-Candy}. The time evolution of the vector function
$\vec{\phi}$ in $t$ for every $x \in \mathbb{R}$ is defined by the linear equation
\begin{equation}
\label{lax-eq-phi-time}
\partial_t\vec{\phi}=A(p,q,\lambda) \vec{\phi}
\end{equation}
for the same $\lambda = e^{i \gamma/2}$. Lemma \ref{Tevol} characterizes two linearly independent solutions
of the linear equation (\ref{lax-eq-phi-time}) for every $t \in \mathbb{R}$.

Lastly, Lemma \ref{lastbd} constructs a new solution $(u,v) \in C(\mathbb{R};H^2(\mathbb{R}))$
of the MTM system (\ref{MTM}) in a $L^2$-neighborhood of the one-soliton solution
from the auto--B\"{a}cklund transformation involving $(p,q)$ and $\vec{\phi}$ for every $t \in \mathbb{R}$.

Let us introduce the following unitary matrices
\begin{equation}\label{unitaryM}
M_1 = \left[\begin{matrix} m_1 & 0 \\ 0 & \overline{m}_1 \end{matrix}\right] \quad \quad
\mbox{and} \quad \quad M_2 = \left[\begin{matrix} \overline{m}_2 & 0 \\ 0 & m_2 \end{matrix}\right],
\end{equation}
where $m_1(x) := e^{\frac{i}{4}\int_{-\infty}^x(|p_0|^2-|q_0|^2)ds}$ and $m_2(x) : = e^{\frac{i}{4}\int_{x}^{\infty}(|p_0|^2-|q_0|^2)ds}$.
We make substitution
\begin{equation}\label{linearsol}
\vec{\phi}_1(x) = e^{-\frac{\sin\gamma}{2}x}M_1(x)\left[\begin{matrix} \varphi_1(x) \\ \varphi_2(x) \end{matrix}\right]
\quad\mbox{and}\quad
\vec{\phi}_2(x) = e^{\frac{\sin\gamma}{2}x}M_2(x)\left[\begin{matrix} \chi_1(x) \\ \chi_2(x) \end{matrix}\right],
\end{equation}
into the linear equation \eqref{lax-eq-phi} with $\lambda = e^{i\gamma/2}$ and obtain two boundary value problems:
\begin{eqnarray}
\label{phieq}
\left\{ \begin{array}{l}
\varphi_1' = \frac{i}{2}(e^{-i\gamma/2}\overline{p}_0-e^{i\gamma/2}\overline{q}_0)\overline{m}_1^2\varphi_2,\\
\varphi_2' = \frac{i}{2}(e^{-i\gamma/2}p_0-e^{i\gamma/2}q_0)m_1^2\varphi_1+\sin\gamma \varphi_2, \end{array} \right.
\end{eqnarray}
and
\begin{eqnarray} \label{chieq}
\left\{ \begin{array}{l}
\chi_1' = -\sin\gamma \chi_1+\frac{i}{2}(e^{-i\gamma/2}\overline{p}_0-e^{i\gamma/2}\overline{q}_0) m_2^2\chi_2, \\
\chi_2' = \frac{i}{2}(e^{-i\gamma/2}p_0 - e^{i\gamma/2}q_0) \overline{m}_2^2 \chi_1, \end{array} \right.
\end{eqnarray}
subject to the boundary conditions
\begin{equation}
\label{phi-chi-bc}
\left\{ \begin{array}{l} \lim\limits_{x\rightarrow -\infty}\varphi_1(x) = 1, \\
\lim\limits_{x\rightarrow \infty}e^{-x\sin\gamma }\varphi_2(x) = 0, \end{array} \right.  \quad\mbox{and}\quad
\left\{ \begin{array}{l} \lim\limits_{x\rightarrow -\infty}e^{x\sin\gamma}\chi_1(x) = 0, \\
\lim\limits_{x\rightarrow \infty}\chi_2(x) = 1.\end{array} \right.
\end{equation}
The following lemma characterizes solutions of the boundary value problems
(\ref{phieq}), (\ref{chieq}), and (\ref{phi-chi-bc}) for small $(p_0,q_0)$ in $L^2(\mathbb{R})$.

\begin{lem} \label{firstbound}
There exists a real positive $\delta$ such that if $\|p_0\|_{L^2}+\|q_0\|_{L^2}\leq \delta$,
then the boundary value problems \eqref{phieq}, (\ref{chieq}), and (\ref{phi-chi-bc}) have
unique solutions in the class
$$
(\varphi_1,\varphi_2)\in L^{\infty}(\mathbb{R}) \times (L^2(\mathbb{R}) \cap L^{\infty}(\mathbb{R})), \quad\mbox{and}\quad
(\chi_1,\chi_2)\in (L^2(\mathbb{R}) \cap L^{\infty}(\mathbb{R})) \times L^{\infty}(\mathbb{R}),
$$
satisfying bounds
\begin{equation}\label{ineqphi}
\|\varphi_1-1\|_{L^{\infty}}+\|\varphi_2\|_{L^2\cap L^{\infty}} \lesssim \|p_0\|_{L^2}+\|q_0\|_{L^2}
\end{equation}
and
\begin{equation}\label{ineqchi}
\|\chi_1\|_{L^2\cap L^{\infty}}+\|\chi_2-1\|_{L^{\infty}} \lesssim \|p_0\|_{L^2}+\|q_0\|_{L^2}.
\end{equation}
\end{lem}

\begin{proof}
The boundary value problem \eqref{phieq} and (\ref{phi-chi-bc}) can be written in the integral form
\begin{equation}
\left\{ \begin{array}{l}
\varphi_1(x) = T_1(\varphi_1,\varphi_2)(x) := 1+\frac{i}{2}\int_{-\infty}^x
\left[e^{-i\gamma/2}\overline{p}_0(y) - e^{i\gamma/2}\overline{q}_0(y) \right] \overline{m}_1^2(y) \varphi_2(y) dy, \\
\varphi_2(x) = T_2(\varphi_1,\varphi_2)(x) := -\frac{i}{2}\int_{x}^{\infty}e^{(x-y)\sin\gamma}
\left[ e^{-i\gamma/2}p_0(y) - e^{i\gamma/2}q_0(y) \right] m_1^2(y) \varphi_1(y) dy. \end{array} \right.
\end{equation}
We introduce a Banach space $Z:= L^{\infty}(\mathbb{R}) \times (L^2(\mathbb{R}) \cap L^{\infty}(\mathbb{R}))$ equipped with the norm
$$
\| \vec{u} \|_{Z}:=\| u_1\|_{L^{\infty}}+\|u_2\|_{L^{\infty}\cap L^2}
$$
and show that $\vec{T} =(T_1,T_2)^t: Z\rightarrow Z$ is a contraction mapping.
Using the Schwartz inequality, the Younge inequality, and the triangle inequality,
we obtain for any $\vec{\varphi}, \widetilde{\vec{\varphi}} \in Z$,
\begin{eqnarray*}
\|T_1(\varphi_1,\varphi_2)-T_1(\widetilde{\varphi}_1,\widetilde{\varphi}_2)\|_{L^{\infty}}
&=&\sup_{x\in \mathbb{R}} \left| \frac{i}{2}\int_{-\infty}^x\left[e^{-i\gamma/2}\overline{p}_0(y)
-e^{i\gamma/2}\overline{q}_0(y) \right]\overline{m}_1^2(y)(\varphi_2(y)-\widetilde{\varphi}_2(y))dy \right|\\
&\leq& \frac{1}{2}\left( \|p_0 \|_{L^2}+ \|q_0 \|_{L^2}\right)\|\varphi_2-\widetilde{\varphi}_2 \|_{L^2}
\end{eqnarray*}
and
\begin{eqnarray*}
\|T_2(\varphi_1,\varphi_2)-T_2(\widetilde{\varphi}_1,\widetilde{\varphi}_2)\|_{L^{\infty}\cap L^2} &\leq&
\frac{1}{2} \|e^{x \sin\gamma} \|_{L^1_x(\mathbb{R}_-)\cap L^2_x(\mathbb{R}_-)}\|e^{-i\gamma/2}p_0-e^{i\gamma/2}q_0 \|_{L^2}\|\varphi_1-\widetilde{\varphi}_1\|_{L^{\infty}}\\
&\leq& \frac{1}{\sin\gamma} \left( \|p_0 \|_{L^2}+ \|q_0 \|_{L^2}\right)\|\varphi_1-\widetilde{\varphi}_1\|_{L^{\infty}}.
\end{eqnarray*}
If $\|p_0 \|_{L^2}+ \|q_0 \|_{L^2} \leq \delta$ is sufficiently small such that $\delta <\sin\gamma$ for a fixed $\gamma \in (0,\pi)$,
then $\vec{T} = (T_1,T_2)^t$ is a contraction mapping on $Z$.
To prove the inequality \eqref{ineqphi}, we have
\begin{eqnarray*}
\|\varphi_1-1\|_{L^{\infty}}+\|\varphi_2\|_{L^2\cap L^{\infty}} &\leq& \| \vec{T}(\varphi_1,\varphi_2)-\vec{T}(0,0)\|_Z \\
&\leq& \frac{\|p_0 \|_{L^2}+ \|q_0 \|_{L^2}}{\sin \gamma} (1+\|\varphi_1-1\|_{L^{\infty}}+\|\varphi_2\|_{L^{\infty}\cap L^2}).
\end{eqnarray*}
Since $\|p_0 \|_{L^2}+ \|q_0 \|_{L^2} \leq \delta < \sin \gamma$, the above estimates yields
the inequality \eqref{ineqphi}. Repeating similar estimates, we can prove that
$(\chi_1,\chi_2)\in (L^2(\mathbb{R}) \cap L^{\infty}(\mathbb{R})) \times L^{\infty}(\mathbb{R})$
and the inequality \eqref{ineqchi} holds.
\end{proof}

Let us now define the time evolution of the vector functions $\vec{\phi}_1$ and $\vec{\phi}_2$
in $t$ for every $x \in \mathbb{R}$, according to the linear equation \eqref{lax-eq-phi-time},
where $\lambda = e^{i \gamma/2}$ and $(p,q)$ are solutions of the MTM system (\ref{MTM})
such that $(p,q) |_{t = 0} = (p_0,q_0)$.
We also consider the initial data for $\vec{\phi}_1$ and $\vec{\phi}_2$ at $t = 0$
given by the two linearly independent solutions \eqref{linearsol} of the linear equation \eqref{lax-eq-phi}.
The linear equation (\ref{lax-eq-phi-time}) for $\vec{\phi}_1$ and $\vec{\phi}_2$ with
$\lambda = e^{i \gamma/2}$ take the form
\begin{equation} \label{phieqt}
\partial_t \vec{\phi}_{1,2} =
\left[ \begin{array}{cc} -\frac{i}{4}(|p|^2+|q|^2)+\frac{i}{2}\cos\gamma &
-\frac{i}{2}(e^{-i\gamma/2}\overline{p}+e^{i\gamma/2}\overline{q}) \\
-\frac{i}{2}(e^{-i\gamma/2}p+e^{i\gamma/2}q) &
\frac{i}{4}(|p|^2+|q|^2)-\frac{i}{2}\cos\gamma \end{array} \right]
\vec{\phi}_{1,2}.
\end{equation}
We set
\begin{equation}
\label{phi-1-2}
\vec{\phi}_1(x,t) := e^{-\frac{x}{2} \sin\gamma} M_1(x,t) \vec{\varphi}(x,t), \quad
\vec{\phi}_2(x,t) := e^{\frac{x}{2} \sin\gamma} M_2(x,t) \vec{\chi}(x,t),
\end{equation}
where $M_1(x,t)$ and $M_2(x,t)$ are given by (\ref{unitaryM}) with
\begin{equation}
\label{m1-m2}
m_1(x,t):= e^{\frac{i}{4}\int_{-\infty}^x(|p(s,t)|^2-|q(s,t)|^2)ds}, \quad
m_2(x,t):= e^{\frac{i}{4}\int_{x}^{\infty} (|p(s,t)|^2-|q(s,t)|^2)ds}.
\end{equation}
The following lemma characterizes vector functions $\vec{\varphi}$ and $\vec{\chi}$.

\begin{lem} \label{Tevol}
Let $(p_0,q_0) \in H^2(\mathbb{R})$ and assume that
there exists a sufficiently small $\delta$ such that
$\|p_0\|_{L^2}+\|q_0\|_{L^2}\leq \delta$. Let $(p,q)\in C(\mathbb{R};H^2(\mathbb{R}))$ be
the unique solution of the MTM system (\ref{MTM}) such that $(p,q) |_{t = 0} = (p_0,q_0)$.
Let $\vec{\phi}_1$ and $\vec{\phi}_2$
be solutions of the linear equation \eqref{phieqt}
starting with the initial data given by solutions of the boundary value problems \eqref{phieq},
\eqref{chieq}, and (\ref{phi-chi-bc}). Then, for every $t\in \mathbb{R}$,
$\vec{\phi}_1$ and $\vec{\phi}_2$ are given by (\ref{phi-1-2}), where
$$
(\varphi_1,\varphi_2)(\cdot,t) \in L^{\infty}(\mathbb{R}) \times (L^2(\mathbb{R}) \cap L^{\infty}(\mathbb{R})) \quad
\mbox{\rm and} \quad
(\chi_1,\chi_2)(\cdot,t) \in (L^2(\mathbb{R}) \cap L^{\infty}(\mathbb{R})) \times L^{\infty}(\mathbb{R})
$$
satisfy the differential equations
\begin{equation}\label{phieq2}
\partial_x \vec{\varphi} =
\left[\begin{matrix} 0 & \frac{i}{2}(e^{-i\gamma/2}\overline{p}-e^{i\gamma/2}\overline{q}) \bar{m}_1^2\\
\frac{i}{2}(e^{-i\gamma/2}p - e^{i\gamma/2}q) m_1^2 & \sin\gamma \end{matrix}\right] \vec{\varphi}
\end{equation}
and
\begin{equation}\label{chieq2}
\partial_x \vec{\chi} =
\left[\begin{matrix} -\sin\gamma & \frac{i}{2}(e^{-i\gamma/2}\overline{p}-e^{i\gamma/2}\overline{q}) m_2^2 \\
\frac{i}{2}(e^{-i\gamma/2}p -e^{i\gamma/2}q) \bar{m}_2^2 & 0 \end{matrix}\right] \vec{\chi},
\end{equation}
subject to the boundary values
\begin{equation}
\label{varphi-bc-chi}
\left\{ \begin{array}{l}
\lim\limits_{x\rightarrow -\infty} \varphi_1(x,t) = e^{i t\frac{\cos\gamma}{2}},\\
\lim\limits_{x\rightarrow \infty}e^{-x\sin\gamma} \varphi_2(x,t) = 0, \end{array} \right. \quad \mbox{\rm and} \quad
\left\{ \begin{array}{l}
\lim\limits_{x\rightarrow -\infty} e^{x\sin\gamma}\chi_1(x,t) = 0,\\
\lim\limits_{x\rightarrow \infty} \chi_2(x,t) = e^{-it\frac{\cos\gamma}{2}}.\end{array} \right.
\end{equation}
Furthermore, for every $t\in \mathbb{R}$, these functions satisfy the bounds
\begin{equation}\label{bdphi}
\|\varphi_1(\cdot,t) - e^{i t \frac{\cos\gamma}{2}}\|_{L^{\infty}}+
\|\varphi_2(\cdot,t) \|_{L^2\cap L^{\infty}} \lesssim \|p_0\|_{L^2}+\|q_0\|_{L^2}
\end{equation}
and
\begin{equation}\label{bdchi}
\|\chi_1(\cdot,t) \|_{L^2\cap L^{\infty}} + \|\chi_2(\cdot,t) -
e^{- i t \frac{\cos\gamma}{2}}\|_{L^{\infty}}\lesssim \|p_0\|_{L^2}+\|q_0\|_{L^2}.
\end{equation}
\end{lem}

\begin{proof}
The unique solution $(p,q)$ of the MTM system (\ref{MTM}) exists in $C(\mathbb{R};H^2(\mathbb{R}))$
by Theorem \ref{theorem-Candy}. By Sobolev embedding of $H^2(\mathbb{R})$ into $C^1(\mathbb{R})$,
the $x$-derivatives of solutions $(p,q)$ are continuous and bounded functions of $x$ for every $t \in \mathbb{R}$.
Moreover, bootstrapping arguments for the MTM system (\ref{MTM}) show that the same solution
$(p,q)$ exists in $C^1(\mathbb{R};H^1(\mathbb{R}))$. Therefore, the $t$-derivatives of solutions
$(p,q)$ are also continuous and bounded functions of $x$ for every $t \in \mathbb{R}$.
The technical assumption $(p_0,q_0) \in H^2(\mathbb{R})$ simplifies working with the system of Lax equations. In particular,
we shall prove that $\vec{\varphi}$ satisfies the differential equation \eqref{phieq2} for every $t\in \mathbb{R}$
if $\vec{\phi}_1$ satisfies the differential equation \eqref{phieqt} for every $x \in \mathbb{R}$.

By Lemma \ref{firstbound}, $\vec{\varphi}$ is a bounded function of $x$ for $t = 0$
and by bootstrapping arguments, $\vec{\varphi} \in C(\mathbb{R})$ for $t = 0$. We now claim that
the differential equation \eqref{phieqt} preserves this property for every $t \in \mathbb{R}$.
From the differential equation \eqref{phieqt} and the representation (\ref{phi-1-2}), we obtain
\begin{eqnarray*}
\partial_t (|\varphi_1|^2+|\varphi_2|^2) & = &
\sin\left(\frac{\gamma}{2}\right) \left[ (\bar{q}-\bar{p}) \bar{m}_1^2 \bar{\varphi}_1 \varphi_2 +
(q-p) m_1^2 \varphi_1 \bar{\varphi}_2\right] \\
&\leq& (|p| + |q|) (|\varphi_1|^2+|\varphi_2|^2).
\end{eqnarray*}
By Gronwall's inequality, for any $T > 0$, we obtain
\begin{equation}\label{Gineq1}
|\varphi_1(x,t)|^2+|\varphi_2(x,t)|^2 \leq e^{\alpha_T T}(|\varphi_1(x,0)|^2+|\varphi_2(x,0)|^2) \quad
x \in \mathbb{R}, \quad t\in [-T,T],
\end{equation}
where
$$
\alpha_T := \sup_{t\in [-T,T]} \sup_{x\in \mathbb{R}} \left( |p(x,t)|+|q(x,t)| \right).
$$
Since the exponential factor remains bounded for any finite time $T>0$,
then it follows that $\vec{\varphi}(\cdot,t) \in L^{\infty}(\mathbb{R})$ for every $t\in \mathbb{R}$.
Bootstrapping then yields $\vec{\varphi}(\cdot,t) \in C(\mathbb{R})$
for every $t \in \mathbb{R}$.

Since coefficients of the linear system (\ref{phieqt}) are continuous functions of $(x,t)$,
we have $\partial_t \vec{\varphi}(\cdot,t) \in C(\mathbb{R})$ for every $t \in \mathbb{R}$.
Now, if $(p,q)$ are $C^1$ functions of $x$ and $t$, then a similar method shows that
$\partial_x \vec{\varphi}, \partial_t \partial_x \vec{\varphi}, \partial^2_t \vec{\varphi} \in C(\mathbb{R})$
for every $t \in \mathbb{R}$.

We shall now establish validity of the differential equation
equation (\ref{phieq2}), which we write in the abstract form as $\partial_x \vec{\phi}_1 = L \vec{\phi}_1$.
We also write the differential equation (\ref{phieqt}) for $\vec{\phi}_1$
in the abstract form as $\partial_t \vec{\phi}_1 = A \vec{\phi}_1$.
To establish (\ref{phieq2}), we construct the residual function
$\vec{F} := \partial_x \vec{\phi}_1 - L \vec{\phi}_1$. This function
is zero for every $x \in \mathbb{R}$ and $t = 0$. We shall prove that
$\vec{F}$ is zero for every $x \in \mathbb{R}$ and $t\in \mathbb{R}$.

The compatibility condition $\partial_x A -\partial_t L + [A,L]=0$
is satisfied for every $x \in \mathbb{R}$ and $t \in \mathbb{R}$,
if $(p,q)$ is a $C^1$ solution of the MTM system (\ref{MTM}).
After differentiating $\vec{F}$ with respect to $t$, we obtain
\begin{eqnarray*}
\partial_t\vec{F} & = & \partial_t\partial_x \vec{\phi}_1 -(\partial_t L) \vec{\phi}_1
- L \partial_t \vec{\phi}_1 \\
& = & \partial_x(A \vec{\phi}_1 ) - (\partial_t L) \vec{\phi}_1 - L A \vec{\phi}_1 \\
& = & (\partial_x A-\partial_t L + [A,L]) \vec{\phi}_1  + A\vec{F} \\
& = & A \vec{F}.
\end{eqnarray*}
Let $\vec{F}=(F_1,F_2)^t$. From the linear evolution $\partial_t\vec{F}= A F$, we again obtain
\begin{eqnarray*}
\partial_t (|F_1|^2+|F_2|^2) & = & \sin\left(\frac{\gamma}{2}\right)
\left[ (\bar{q}-\bar{p}) \bar{F}_1 F_2 + (q-p) F_1 \bar{F}_2\right] \\
&\leq& (|p|+|q|)(|F_1|^2+|F_2|^2),
\end{eqnarray*}
which yields with Gronwall's inequality for any $T>0$
$$
|F_1(x,t)|^2+|F_2(x,t)|^2 \leq e^{\alpha_T T}
(|F_1(x,0)|^2+|F_2(x,0)|^2), \quad x \in \mathbb{R}, \quad t\in [-T,T],
$$
with the same definition of $\alpha_T$.
Since $\vec{F}(x,0)=\vec{0}$, then the above inequality yields $\vec{F}(x,t) = \vec{0}$
for every $x \in \mathbb{R}$ and $t \in [-T,T]$. Hence,
$\vec{\varphi}$ satisfies the differential equation \eqref{phieq2}.

We have shown that $\vec{\varphi}(\cdot,t) \in L^{\infty}(\mathbb{R})$ for every $t \in \mathbb{R}$.
We now show that $\varphi_2(\cdot,t) \in L^2(\mathbb{R})$ for every $t\in \mathbb{R}$.
It follows from the differential equation \eqref{phieqt} and the representation (\ref{phi-1-2}) that
\begin{eqnarray*}
\partial_t(|\varphi_2 |^2) & \leq & (|p|+|q|) |\bar{\varphi}_1 \varphi_2|\\
&\lesssim & |\varphi_2|^2+(|p|^2+|q|^2)|\varphi_1|^2.
\end{eqnarray*}
Using Gronwall's inequality and the previous bound \eqref{Gineq1}, we have
for any $T>0$
\begin{eqnarray*}
|\varphi_2(x,t)|^2&\leq& e^{T}\left[|\varphi_2(x,0)|^2 +
\int_{-T}^{T}(|p(x,s)|^2+|q(x,s)|^2)|\varphi_1(x,s)|^2ds\right]\\
&\leq& e^T |\varphi_2(x,0)|^2 + e^{(1+\alpha_T) T}
\int_{-T}^{T}(|p(x,s)|^2+|q(x,s)|^2) \left( |\varphi_1(x,0)|^2 + |\varphi_2(x,0)|^2 \right) ds,
\end{eqnarray*}
where $x \in \mathbb{R}$ and $t \in [-T,T]$. Therefore, we have
$$
\|\varphi_2(\cdot,t)\|_{L^2}^2 \leq e^T \|\varphi_2(\cdot,0)\|_{L^2}^2 + e^{(1+\alpha_T) T}
\left( \|\varphi_1(\cdot,0)\|_{L^{\infty}}^2
+ \| \varphi_2(\cdot,0)\|_{L^{\infty}}^2 \right) \int_{-T}^{T}(\|p(\cdot,s)\|_{L^2}^2+\|q(\cdot,s)\|_{L^2}^2)ds .
$$
Since the right-hand side of this inequality remains bounded for any
finite time $T>0$, then it follows that $\varphi_2(\cdot,t) \in L^2(\mathbb{R})$ for every $t\in \mathbb{R}$.

It remains to prove the boundary values for $\vec{\varphi}_1(x,t)$ as $x \to \pm \infty$
in \eqref{varphi-bc-chi}. The second boundary condition
$$
\lim_{x\rightarrow \infty}e^{-x\sin\gamma}\varphi_2(x,t)=0
$$
follows from the fact
that $\varphi_2(\cdot,t)\in L^{\infty}(\mathbb{R})$ for every $t\in \mathbb{R}$.
To prove the first boundary condition, we use Duhamel's formula to write
the differential equation (\ref{phieqt}) in the integral form:
$$
\vec{\phi}_1(x,t) = e^{i t \sigma_3 \frac{\cos\gamma}{2}} \vec{\phi}_1(x,0) +
\int_0^t e^{i (t-s)\sigma_3\frac{\cos\gamma}{2}} A_1(x,s) \vec{\phi}_1(x,s)ds,
$$
where
$$
A_1(x,t) := \left[ \begin{array}{cc} -\frac{i}{4}(|p|^2+|q|^2)  &
-\frac{i}{2}(e^{-i\gamma/2}\overline{p}+e^{i\gamma/2}\overline{q}) \\
-\frac{i}{2}(e^{-i\gamma/2}p+e^{i\gamma/2}q) & \frac{i}{4}(|p|^2+|q|^2) \end{array} \right].
$$
Using the representation (\ref{phi-1-2}), we have for $t\in \mathbb{R}$
$$
|M_1 \vec{\varphi}(x,t)-e^{i t \sigma_3 \frac{\cos\gamma}{2}} M_1 \vec{\varphi}(x,0)| \leq
\int_0^{|t|}|A_1(x,s) M_1 \vec{\varphi}(x,s)|ds,
$$
where $|\vec{f}|$ denotes the vector norm of the $2$-vector $\vec{f}$.
Since $\vec{\varphi}(\cdot,t)\in L^{\infty}(\mathbb{R}) \times (L^{\infty}(\mathbb{R}) \cap L^2(\mathbb{R}))$
for every $t\in \mathbb{R}$ and $p(\cdot,t),q(\cdot,t)\in H^2(\mathbb{R})$,
we claim that
\begin{itemize}
\item $|A_1(x,s) M_1 \vec{\varphi}(x,s)|$ is bounded by some $s$-independent constant for
every $x\in \mathbb{R}$ and $|s| \leq |t|$

\item $\lim_{|x|\rightarrow \infty} A_1(x,s) M_1 \vec{\varphi}(x,s) =\vec{0}$ pointwise for every $|s| \leq |t|$.
\end{itemize}
Then, the dominated convergence theorem gives
$$
\lim_{x\rightarrow -\infty} |M_1(x,t) \vec{\varphi}(x,t) - e^{i t \sigma_3 \frac{\cos\gamma}{2} }
M_1(x,0) \vec{\varphi}(x,0)| = 0,
\quad t \in \mathbb{R}.
$$
Since $\vec{\varphi}(x,0) \rightarrow (1,0)^t$ as $x\rightarrow -\infty$ and
$M_1(x,t) \to I$ as $x \rightarrow -\infty$ for every $t \in \mathbb{R}$,
the above limit recovers the first boundary condition
$$
\lim_{x\rightarrow -\infty}\varphi_1(x,t) = e^{i t \sigma_3 \frac{\cos\gamma}{2}}.
$$

The proof of the differential equation (\ref{chieq2}) and the boundary condition
for $\vec{\chi}$ in \eqref{varphi-bc-chi} is analogous. Finally, since the
$L^2$ norm of solutions of the MTM system (\ref{MTM}) is constant, according to
the conserved law \eqref{l2-conservation}, the proof of bounds \eqref{bdphi} and
\eqref{bdchi} is analogous to the proof in Lemma \ref{firstbound}.
\end{proof}

\begin{lem} \label{lastbd}
Let $(p_0,q_0) \in H^2(\mathbb{R})$ and assume that there exists a sufficiently small $\delta$ such that
$\|p_0\|_{L^2}+\|q_0\|_{L^2}\leq \delta$. Let $(p,q) \in C(\mathbb{R};H^2(\mathbb{R}))$ be
the unique solution of the MTM system (\ref{MTM}) such that $(p,q) |_{t = 0} = (p_0,q_0)$. Using
solutions $\vec{\varphi}$ and $\vec{\chi}$ in Lemma \ref{Tevol}, let us define
\begin{equation}\label{sol1}
\left[ \begin{array}{l} \phi_1(x,t) \\ \phi_2(x,t) \end{array} \right] :=
c_1(t) e^{-\frac{x}{2} \sin\gamma}  M_1(x,t) \vec{\varphi}(x,t) + c_2(t) e^{\frac{x}{2} \sin\gamma} M_2(x,t) \vec{\chi}(x,t),
\end{equation}
where $c_1(t) := e^{(a+i\theta)/2}$, $c_2(t) := e^{-(a+i\theta)/2}$ are given in terms 
of the real coefficients $a, \theta$, which may depend on $t$. Then, the auto--B\"{a}cklund transformation
\begin{equation}\label{d1u}
u := - p \frac{e^{-i\gamma/2}|\phi_1|^2+e^{i\gamma/2}|\phi_2|^2}{e^{i\gamma/2}|\phi_1|^2+e^{-i\gamma/2}|\phi_2|^2}+
\frac{2i \sin\gamma \overline{\phi}_1\phi_2}{e^{i\gamma/2}|\phi_1|^2+e^{-i\gamma/2}|\phi_2|^2}
\end{equation}
and
\begin{equation} \label{d1v}
v := -q \frac{e^{i\gamma/2}|\phi_1|^2+e^{-i\gamma/2}|\phi_2|^2}{e^{-i\gamma/2}|\phi_1|^2+e^{i\gamma/2}|\phi_2|^2}-
\frac{2i \sin\gamma \overline{\phi}_1\phi_2}{e^{-i\gamma/2}|\phi_1|^2+e^{i\gamma/2}|\phi_2|^2}
\end{equation}
generates a new solution $(u,v) \in C(\mathbb{R};H^2(\mathbb{R}))$ of the MTM system (\ref{MTM}) satisfying the bound
\begin{equation}\label{ineqU}
\sup_{t\in\mathbb{R}} \left\| u(x,t) - ie^{-i\theta-it\cos\gamma}\sin\gamma \sech \left( x \sin \gamma - i\frac{\gamma}{2}-a \right)
\right\|_{L^2_x} \lesssim \|p_0 \|_{L^2}+ \|q_0 \|_{L^2}
\end{equation}
and
\begin{equation}\label{ineqV}
\sup_{t\in\mathbb{R}} \left\| v(x,t) + ie^{-i\theta-it\cos\gamma}\sin\gamma \sech \left( x \sin \gamma + i\frac{\gamma}{2}-a \right)
\right\|_{L^2_x} \lesssim \|p_0 \|_{L^2}+ \|q_0 \|_{L^2}.
\end{equation}
\end{lem}

\begin{proof}
Let us introduce $\vec{\psi}=(\psi_1,\psi_2)^t$ by
\begin{equation}\label{d2}
\psi_1:=\frac{\overline{\phi}_2}{|e^{i\gamma/2}|\phi_1|^2+e^{-i\gamma/2}|\phi_2|^2|}, \quad \psi_2:=\frac{\overline{\phi}_1}{|e^{i\gamma/2}|\phi_1|^2+e^{-i\gamma/2}|\phi_2|^2|}.
\end{equation}
The inequalities \eqref{bdphi} and \eqref{bdchi} imply that $(u,v)$ and $\vec{\psi}$
are bounded for every $x \in \mathbb{R}$ and $t \in \mathbb{R}$. 
If $(p,q)$ are $C^1$ functions of $(x,t)$ and $\vec{\phi}$ is a $C^2$ function of $(x,t)$, 
then $(u,v)$ are $C^1$ functions of $(x,t)$ and 
$\vec{\psi}$ is a $C^2$ function of $(x,t)$.
Proposition \ref{backlund} states that
$\vec{\psi}$ given by \eqref{d2} satisfies the evolution equations
$$
\partial_x\vec{\psi}= L(u,v,\lambda) \vec{\psi}, \quad
\partial_t\vec{\psi} = A(u,v,\lambda) \vec{\psi},
$$
for $\lambda = e^{i \gamma/2}$. As a result, the compatibility condition
$\partial_x\partial_t\vec{\psi}=\partial_t\partial_x\vec{\psi}$
for every $x \in \mathbb{R}$ and $t \in \mathbb{R}$ yields the MTM
system (\ref{MTM}) for the functions $(u,v)$.

We shall now prove inequality \eqref{ineqU}. The proof of inequality \eqref{ineqV}
is analogous. First, we write \eqref{d1u} in the form of
\begin{equation} \label{Req}
R := \frac{2i\sin\gamma \overline{\phi}_1\phi_2}{e^{i\gamma/2}|\phi_1|^2+e^{-i\gamma/2}|\phi_2|^2} =
u + p \frac{e^{-i\gamma/2}|\phi_1|^2+e^{i\gamma/2}|\phi_2|^2}{e^{i\gamma/2}|\phi_1|^2+e^{-i\gamma/2}|\phi_2|^2}.
\end{equation}
Explicit substitutions of \eqref{sol1} into \eqref{Req} yield
$$
R := \frac{2i\sin\gamma \left( \overline{m}_1m_2e^{-i\theta}\overline{\varphi}_1\chi_2+R_1\right)}{
e^{i\gamma/2+a - x \sin\gamma}|\varphi_1|^2+e^{-i\gamma/2-a + x \sin\gamma}|\chi_2|^2+R_2},
$$
where
$$
R_1 := \overline{m}_1^2e^{a- x \sin\gamma}\overline{\varphi}_1\varphi_2+
\overline{m}_1m_2 e^{i\theta} \varphi_2\overline{\chi}_1+m_2^2e^{-a + x \sin\gamma}\overline{\chi}_1\chi_2
$$
and
$$
R_2 := e^{i\gamma/2-a + x \sin\gamma}|\chi_1|^2+e^{-i\gamma/2+a - x \sin\gamma}
|\varphi_2|^2+2e^{i\gamma/2}\mbox{Re}[m_1m_2e^{i\theta}\varphi_1\overline{\chi}_1]+
2 e^{-i\gamma/2} \mbox{Re}[\overline{m}_1\overline{m}_2e^{i\theta}\varphi_2\overline{\chi}_2].
$$

By bounds \eqref{bdphi} and \eqref{bdchi} in Lemma \ref{Tevol}, we have
$|\varphi_1|, |\chi_2|\sim 1$ and $|\varphi_2|, |\chi_1|\sim 0$,
so that for $a-x\sin\gamma \leq 0$,
\begin{equation} \label{Rp}
R = \frac{2i\sin\gamma \overline{m}_1m_2e^{-i\theta+a-x\sin\gamma}
\overline{\varphi}_1\chi_2}{e^{i\gamma/2+2(a-x\sin\gamma)}|\varphi_1|^2+e^{-i\gamma/2}|\chi_2|^2}+
\mathcal{O}(|\varphi_2| + |\chi_1|)
\end{equation}
and for $a-x\sin\gamma \geq 0$,
\begin{equation}\label{Rn}
R = \frac{2i\sin\gamma \overline{m}_1m_2e^{-i\theta-a+x\sin\gamma}
\overline{\varphi}_1\chi_2}{e^{i\gamma/2}|\varphi_1|^2+e^{-i\gamma/2-2(a-x\sin\gamma)}|\chi_2|^2}+
\mathcal{O}(|\varphi_2| + |\chi_1|).
\end{equation}
Combining \eqref{Rp} and \eqref{Rn}, we get
\begin{eqnarray*}
& \phantom{t} & \left| R - \frac{2i\sin\gamma e^{-i\theta-it\cos\gamma}}{
e^{i\gamma/2+a-x\sin\gamma} + e^{-i\gamma/2}-a + x \sin \gamma} \right|\\
&\lesssim& e^{-|a-x\sin\gamma|}(|\varphi_1 - e^{it\frac{\cos\gamma}{2}}|
+ |\chi_2 - e^{-it\frac{\cos\gamma}{2}}|+|m_1-1| + |m_2-1|)+|\varphi_2|+|\chi_1|
\end{eqnarray*}
Since $m_1=e^{\frac{i}{4} \int_{-\infty}^x(|p|^2-|q|^2)ds}$ and $m_2=e^{\frac{i}{4}\int_x^{\infty}(|p|^2-|q|^2)ds}$,
we obtain the bounds
$$
\| m_{1,2}(\cdot,t) - 1 \|_{L^{\infty}} \lesssim \|p\|_{L^2}^2 + \|q\|_{L^2}^2,
$$
provided that $\|p\|_{L^2}$ and $\|q\|_{L^2}$ are sufficiently small.
Then, by Lemma \ref{Tevol} and the $L^2$ conservation law (\ref{l2-conservation}), the previous estimate yields
\begin{equation} \label{ineqR}
\left\| R(x,t)-ie^{-i\theta-it\cos\gamma}\sin\gamma \sech \left( x \sin\gamma -i\frac{\gamma}{2}-a \right)
\right\|_{L^2_x} \lesssim \|p_0\|_{L^2}+\|q_0\|_{L^2}.
\end{equation}
Using the definition (\ref{Req}), the bound \eqref{ineqR}, and the triangle inequality, we obtain inequality \eqref{ineqU}.

Lastly, if $(p,q) \in C(\mathbb{R};H^2(\mathbb{R}))$, we can differentiate equation \eqref{Rp} and \eqref{Rn} in $x$ twice
to show from \eqref{d1u} and \eqref{d1v} that $(u,v) \in C(\mathbb{R};H^2(\mathbb{R}))$.
\end{proof}

\section{Proof of Theorem \ref{maintheorem}}

Thanks to the Lorentz transformation given by Proposition \ref{prop-Lorentz},
we may choose $\lambda_0=e^{i\gamma_0/2}$, $\gamma_0\in(0,\pi)$ in Theorem \ref{maintheorem}.
For a given initial data $(u_0,v_0)$ satisfying the inequality (\ref{bound-before}) for
sufficiently small $\epsilon$, we map a $L^2$-neighborhood of one-soliton solution to that of the zero solution.
To do so, we use Lemma \ref{onebound1} and obtain a solution $\vec{\psi}$
of the linear system \eqref{laxeq2} with a spectral parameter $\lambda\in \mathbb{C}$ satisfying
\begin{equation}\label{ineq1}
|\lambda-e^{i\gamma_0/2}|\lesssim  \|u_0 - u_{\gamma_0}\|_{L^2}+\|v_0-v_{\gamma_0}\|_{L^2} \leq \epsilon.
\end{equation}
We should note that the same Lorentz transformation cannot be used twice to consider the cases
of $\lambda_0=e^{i\gamma_0/2}$ and $\lambda=e^{i\gamma/2}$ simultaneously;
the assumption $\lambda_0 = e^{i\gamma_0/2}$ implies that $\lambda$ is not generally
on the unit circle, and vice versa. Hence, if $\lambda_0 = e^{i \gamma_0/2}$ is set,
all formulas in Section 3 below Remark \ref{remark-normalization} must in fact be generalized for
a general $\lambda$. However, this generalization is straightforward thanks again to the existence
of the Lorentz transformation given by Proposition \ref{prop-Lorentz}.
In what follows, we then use the general MTM solitons $(u_{\lambda},v_{\lambda})$ given by (\ref{onesol}).

By Lemma \ref{pqbound}, the auto--B\"{a}cklund transformation \eqref{up}
and \eqref{vq} with $\vec{\psi}$ in Lemma \ref{onebound1}
yields an initial data $(p_0,q_0) \in L^2(\mathbb{R})$ of the MTM
system (\ref{MTM}) satisfying the estimate
\begin{eqnarray}
\nonumber
\|p_0\|_{L^2} +\|q_0\|_{L^2} & \lesssim &
\|u_0 - u_{\lambda}(\cdot,0) \|_{L^2}+\|v_0-v_{\lambda}(\cdot,0)\|_{L^2} \\
\nonumber
& \lesssim & \|u_0 - u_{\gamma_0}\|_{L^2}+\|v_0-v_{\gamma_0}\|_{L^2}
+ \|u_{\lambda}(\cdot,0) - u_{\gamma_0}\|_{L^2}+\|v_{\lambda}(\cdot,0) - v_{\gamma_0}\|_{L^2} \\
\label{ineq2}
& \lesssim & \|u_0 - u_{\gamma_0}\|_{L^2}+\|v_0-v_{\gamma_0}\|_{L^2}  \leq \epsilon,
\end{eqnarray}
where we have used the triangle inequality and the bound (\ref{ineq1}).

Since the time evolution in Section 4 is well-defined if $(p_0,q_0) \in H^2(\mathbb{R})$,
let us first assume that the initial data $(u_0,v_0) \in L^2(\mathbb{R})$
satisfying the inequality (\ref{bound-before}) also satisfy $(u_0,v_0) \in H^2(\mathbb{R})$.
Then, $(p_0,q_0) \in H^2(\mathbb{R})$ by Lemma \ref{pqbound}.
Let $(p,q) \in C(\mathbb{R};H^2(\mathbb{R}))$ be the unique solution of the MTM system (\ref{MTM})
such that $(p,q) |_{t = 0} = (p_0,q_0)$. Next we will map a $L^2$-neighborhood
of the zero solution to that of one-soliton solution for all $t\in \mathbb{R}$.

By Lemma \ref{Tevol}, we construct a solution of the Lax equations
\begin{equation}
\label{lax-system}
\partial_x \vec{\phi} = L(p,q,\lambda) \vec{\phi} \quad
\mbox{\rm and} \quad \partial_t \vec{\phi} = A(p,q,\lambda)
\vec{\phi}
\end{equation}
for the same eigenvalue $\lambda$ as in (\ref{ineq1}). Let
$$
k_1(\lambda) := \frac{i}{4} \left( \lambda^2 - \frac{1}{\lambda^2} \right), \quad
k_2(\lambda) := \frac{1}{4} \left( \lambda^2 + \frac{1}{\lambda^2} \right).
$$
The solution of the Lax system (\ref{lax-system}) is constructed in the form
\begin{equation}
\label{sol1-final}
\vec{\phi}(x,t) = c_1 M_1(x,t) e^{x k_1(\lambda)} \vec{\varphi}(x,t)
+ c_2 M_2(x,t) e^{-x k_1(\lambda)} \vec{\chi}(x,t),
\end{equation}
where unitary matrices $M_1$ and $M_2$ are given in \eqref{unitaryM}
with $m_1$ and $m_2$ given by (\ref{m1-m2}), whereas
the vectors $\vec{\varphi}$ and $\vec{\chi}$ satisfy the estimates
\begin{equation} \label{est1}
\|\varphi_1(\cdot,t) - e^{it k_2(\lambda)}\|_{L^{\infty}} +
\|\varphi_2(\cdot,t) \|_{L^2 \cap L^{\infty}} \lesssim \|p_0\|_{L^2}+\|q_0\|_{L^2} \leq \epsilon
\end{equation}
and
\begin{equation}\label{est2}
\|\chi_1(\cdot,t)\|_{L^2 \cap L^{\infty}} + \|\chi_2(\cdot,t) - e^{-itk_2(\lambda)}\|_{L^{\infty}}
\lesssim \|p_0\|_{L^2}+\|q_0\|_{L^2} \leq \epsilon.
\end{equation}
The coefficients $c_1$ and $c_2$  of the linear superposition (\ref{sol1-final})
can be parameterized by parameters $a$ and $\theta$ as follows:
$$
c_1 = e^{(a+i\theta)/2}, \quad c_2 = e^{-(a+i\theta)/2}.
$$
Parameters $a$ and $\theta$ may depend on the time variable $t$ but not on the space
variable $x$. These parameters determine the spatial and gauge translations of the MTM
solitons according to the transformation (\ref{Soliton-Transformation}).

By Lemma \ref{lastbd}, the auto--B\"{a}cklund transformation generates a new solution $(u,v)$ of the
MTM system (\ref{MTM}) satisfying the bound for every $t\in \mathbb{R}$,
\begin{equation}
\label{ineq3}
\inf_{a,\theta\in \mathbb{R}}(\|u(\cdot+a,t)-e^{-i\theta}u_{\lambda}(\cdot,t) \|_{L^2}
+\|v(\cdot+a,t)-e^{-i\theta}v_{\lambda}(\cdot,t)\|_{L^2})\lesssim \|p_0\|_{L^2}+\|q_0\|_{L^2} \leq \epsilon.
\end{equation}

Theorem \ref{maintheorem} is proved if $(u_0,v_0) \in H^2(\mathbb{R})$.
To obtain the same result for $(u_0,v_0) \in L^2(\mathbb{R})$, we
construct an approximating sequence $(u_{0,n},v_{0,n}) \in H^2(\mathbb{R})$ $(n\in \mathbb{N})$
that converges as $n \to \infty$ to $(u_0,v_0) \in L^2(\mathbb{R})$ in the $L^2$-norm.
For a sufficiently small $\epsilon>0$, we let
$$
\|u_{0,n}-u_{\gamma_0}\|_{L^2}+\|v_{0,n}-v_{\gamma_0}\|_{L^2} \leq \epsilon, \quad
\mbox{\rm for every} \; n \in \mathbb{N}.
$$
Under this condition, for each $(u_{0,n},v_{0,n}) \in H^2(\mathbb{R})$,
we obtain inequalities \eqref{ineq1}, \eqref{ineq2}, and \eqref{ineq3}
independently of $n$. Therefore, there is a subsequence of solutions
$(u_{n},v_{n}) \in C(\mathbb{R};H^2(\mathbb{R}))$ $(n\in \mathbb{N})$
of the MTM system (\ref{MTM}) such that
it converges as $n \to \infty$ to a solution $(u,v) \in C(\mathbb{R};L^2(\mathbb{R}))$
of the MTM system (\ref{MTM}) satisfying inequalities (\ref{bound-parameters}) and (\ref{bound-after}).
The proof of Theorem \ref{maintheorem} is now complete.

\end{document}